\documentclass[letterpaper,11pt]{amsart}
\usepackage{tikz, tikz-cd}
\usepackage{pgfplots}
\usetikzlibrary{arrows}
\usepackage{subcaption} 
\usepackage[font=footnotesize,labelfont=bf]{caption}
\usepackage{blkarray}
\usepackage{enumitem}
\usepackage{hyperref}

\usepackage{latexsym,array,delarray,epsfig,setspace,cleveref,mathtools,amssymb,mathrsfs}
\usepackage{amsfonts,amssymb,amsmath,amsthm,pgfplots,tikz,tikz-cd,float,tkz-euclide}
\pgfplotsset{compat=1.16}

\definecolor{light}{gray}{.75}
\definecolor{med}{gray}{.5}
\definecolor{dark}{gray}{.25}

\newtheorem{theorem}{Theorem}
\numberwithin{theorem}{section}
\newtheorem{proposition}[theorem]{Proposition}

\newtheorem{corollary}[theorem]{Corollary}
\newtheorem{lemma}[theorem]{Lemma}
\newtheorem{conjecture}[theorem]{Conjecture}
\newtheorem{question}[theorem]{Question}
\theoremstyle{definition}
\newtheorem{definition}[theorem]{Definition}
\newtheorem{remark}[theorem]{Remark}
\newtheorem{example}[theorem]{Example}
\newcommand{\C}{{\mathbb C}}

\newcommand{\Q}{{\mathbb Q}}
\newcommand{\R}{{\mathbb R}}
\newcommand{\RR}{\mathcal{R}}
\newcommand{\Z}{{\mathbb Z}}
\renewcommand{\P}{{\mathbb P}}

\newcommand{\MM}{\mathcal{M}}

\newcommand{\B}{\mathbb{B}}
\newcommand{\E}{\mathcal{E}}
\newcommand{\F}{\mathcal{F}}

\newcommand{\PP}{\mathcal{P}}
\renewcommand{\O}{\mathcal{O}}
\newcommand{\V}{\mathcal{V}}

\renewcommand{\L}{\mathcal{L}}

\newcommand{\Trop}{\textup{Trop}}
\newcommand{\T}{\mathcal{T}}
\newcommand{\I}{\mathcal{I}}

\newcommand{\be}{\mathbf{e}}
\newcommand{\bee}{\tilde{\mathbf{e}}}
\newcommand{\bd}{\mathbf{d}}

\newcommand{\br}{\mathbf{r}}

\newcommand{\bt}{\mathbf{t}}
\newcommand{\ba}{\mathbf{a}}
\newcommand{\bc}{\mathbf{c}}
\newcommand{\nn}{\underline{n}}
\newcommand{\bu}{\mathbf{u}}

\newcommand{\In}{\textup{in}}

\renewcommand{\v}{\mathfrak{v}}

\renewcommand{\In}{\textup{in}}

\newcommand{\MIN}{\textup{MIN}}
\newcommand{\trop}{\textup{trop}}

\newcommand{\Proj}{\textup{Proj}}
\newcommand{\Pic}{\textup{Pic}}
\newcommand{\CL}{\textup{CL}}

\newcommand{\GF}{\textup{GF}}

\newcommand{\Sym}{\textup{Sym}}

\newcommand{\Hom}{\textup{Hom}}
\newcommand{\Eff}{\textup{Eff}}
\newcommand{\PsEff}{\textup{PsEff}}
\newcommand{\Nef}{\textup{Nef}}
\newcommand{\Bpf}{\textup{Bpf}}

\title{Positivity properties of divisors on Toric Vector Bundles}
\author{Courtney George, Christopher Manon}

\begin{document}

\maketitle

\begin{abstract}
We use presentations of the Cox rings of projectivized toric vector bundles and elements of matroid theory to compute Newton-Okounkov bodies, effective cones, and Nef cones of these spaces. As an application we analyze the Fano property, and establish Fujita's freeness and ampleness conjectures for several classes of projectivized toric vector bundles. 
\end{abstract}

\tableofcontents

\section{Introduction }

Let $N \cong \Z^d$ be a lattice with dual lattice $M = \Hom(N, \Z)$, and let $T$ denote the algebraic torus with cocharacter lattice $N$.  Let $\Sigma \subset N\otimes \Q$ be a complete, polyhedral fan, and let $X(\Sigma)$ denote the corresponding complete toric variety over $\C$. We assume throughout that $X(\Sigma)$ is a projective variety. A \emph{toric vector bundle} $\pi: \E \to X(\Sigma)$ is a vector bundle equipped with a linear action by $T$ which intertwines through $\pi$ with the $T$ action on $X(\Sigma)$.  

Toric vector bundles are classified by Kaneyama in \cite{Kaneyama}, and by Klyachko in  \cite{Klyachko}. The Klyachko data of toric vector bundle consists of a collection $F^{\rho}$ of compatible integral filtrations in the identity fiber $\pi^{-1}Id = E$ of $\E$, one for each ray $\rho \in \Sigma(1)$.  In this way, toric vector bundles admit a combinatorial classification, much like toric varieties.  In addition to the customary polyhedral data that comes with $X(\Sigma)$, we have the linear combinatorics of matroids data living in the arrangement of Klyachko's spaces.  This theme repeats for other invariants of toric vector bundles as well, for example the global section space $H^0(X(\Sigma), \E)$ can be read off the so-called \emph{parliament of polytopes} of DiRocco, Jabbusch, and Smith \cite{DJS}, which can be viewed as a labelling of a matroid by polyhedra.  

The projectivization $\P\E$ of a toric vector bundle is a smooth, projective variety with a free class group. The Cox ring $\RR(X(\Sigma))$ (see \cite{ADUH-book}) of a toric variety is well-known to be a polynomial ring in $n = |\Sigma(1)|$ variables. In contrast, the Cox rings of projectivized toric vector bundles are a much richer class of objects.  For example, by results of Gonz\'alez, Hering, Payne, and S\"u\ss, both the Cox ring of $\bar{M}_{0,n}$ and the additive invariant rings of Nagata can be realized as Cox rings of certain projectivized toric vector bundles. In particular, it is possible for $\RR(\P\E)$ to be infinitely generated. In the case that $\RR(\P\E)$ is finitely generated, $\P\E$ is called a \textit{Mori dream space} (see \cite{Hu-Keel}). 

Cox rings of toric vector bundles are shown in \cite{Kaveh-Manon-tvb} to be \emph{multi-Rees algebras} associated to a collection filtrations on a polynomial ring coming from a tropicalized linear space.  In \cite{Kaveh-Manon-tvb} algebraic and combinatorial criteria for this type of ring to finitely generated, along with an algorithm to construct a generating set, are given.  In this paper we use these descriptions of $\RR(\P\E)$, along with aspects of associated matroids, to study the convex geometry associated to toric vector bundles and their projectivizations.  

\subsection{Newton-Okounkov bodies}

Thoughout the paper we focus on projectivized toric vector bundles $\P\E$ whose Cox ring $\RR(\P\E)$ is generated in minimal degree ($\Sym$-degree $1$, see Section \ref{sec-background}), although we do examples (Example \ref{ex-Hilb2}) where higher degree generators appear. 

In Section \ref{sec-background} we recall work of Kaveh and the second author \cite{Kaveh-Manon-tvb}, where it is shown that a toric vector bundle $\E$ can be constructed from a choice of linear ideal $L \subset \C[y_1, \ldots, y_m]$ and an $n \times m$ matrix $D$ called a \emph{diagram}.  The rows of the diagram $D$ are points in the tropical variety $\Trop(L)$, and the columns $D_j$ each determine a divisor class $\bd_j \in \CL(X(\Sigma))$. The linear ideal determines a (representable) matroid $\MM(L)$ on the elements $y_1, \ldots, y_m$. 

We use tropical techniques of \cite{Kaveh-Manon-NOK} to construct global Newton-Okounkov bodies of $\P\E$, along with Newton-Okounkov bodies of complete linear series on $\P\E$. 
Before we state our main result, we introduce some terminology. The tropical variety $\Trop(L)$ can be identified with the Bergman fan of of $\MM(L)$.  A maximal face $K \subset \Trop(L)$ corresponds to a flag $F_1 \subset \cdots \subset F_r$ of flats in $\MM(L)$, we let $E_K$ be the $r \times m$ matrix with rows the indicator vectors of these flats. Finally, we let:\\

\[M = \begin{bmatrix}
     D & -I\\
     E_K & 0
     \end{bmatrix}.
\]\\

\noindent
For each divisor class $(\alpha, \beta) \in \CL(\P\E) \cong \CL(X(\Sigma))\times \Z$ we define a certain polytope $P_{\alpha, \beta} \subset \Q_{\geq 0}^{n + m}$, which we show is naturally a Cayley polytope in nice cases, see Section \ref{sec-noby}. Finally, a cone in a tropical variety is said to be a \emph{prime cone} if its associated initial ideal is a prime ideal (see \cite{Kaveh-Manon-NOK}). The following is our main theorem on Newton-Okounkov bodies of projectivized toric vector bundles. 

\begin{theorem}\label{thm-mainnoby}
Let $\E$ be a toric vector bundle with $\RR(\P\E)$ generated in $\Sym$-degree 1, and suppose that there is a maximal flag of flats $F_1 \subset \cdots \subset F_r$ in the matroid $\MM(L)$ such that the inverse image of the corresponding maximal face of the Bergman fan $\Trop(L)$ is a prime cone of $\Trop(\I)$, then:\\

\begin{enumerate}
\item $\Delta = M\circ \Q_{\geq 0}^{m + n}$ is a global Newton-Okounkov body of $\P\E$,
\item $\Delta_{\alpha, \beta} = M\circ P_{\alpha, \beta}$ is the Newton-Okounkov body of $(\alpha, \beta) \in \CL(\P\E)$.\\
\end{enumerate}

\end{theorem}

We describe several large families of projectivized toric vector bundles which satisfy the conditions of Theorem \ref{thm-mainnoby}. As two applications, we construct integral polytopes whose lattice points enumerate the global sections of $\E$ itself, and we construct the Newton-Okounkov bodies of any tangent bundle $\T X(\Sigma)$ of a smooth toric variety $X(\Sigma)$. 

\subsection{The Nef Cone}

Next we describe the Nef cone of $\P\E$ in terms of the flats of the \emph{initial matroids} $\MM(\In_\sigma L)$ (see Section \ref{sec-positivity}). For each facet $\sigma \in \Sigma$ there is an associated initial ideal $\In_\sigma L \subset \C[y_1, \ldots, y_m]$. We let $F$ denote a maximal non-trivial flat of the matroid $\MM(\In_\sigma L)$. For each pair $\sigma \in \Sigma$, $F \subset \MM(\In_\sigma L)$ we define a certain point $p_{\sigma, F} \in \P\E$, with an associated explicit monoid $S_{\sigma, F} \subset \CL(\P\E)$ and cone $C_{\sigma, F} \subset \CL(\P\E)\otimes \Q$. 

\begin{theorem}\label{thm-main-nef}
Let $\E$ be a toric vector bundle with $\RR(\P\E)$ generated in $\Sym$-degree 1, then

\[\Bpf(\P\E) = \bigcap_{\sigma, F} S_{\sigma, F},\]

\[\Nef(\P\E) = \bigcap_{\sigma, F} C_{\sigma, F}.\]

\end{theorem}

We make use Theorem \ref{thm-main-nef} by imposing conditions on $\MM(L)$ and the $\MM(\In_\sigma L)$ which simplify the structure of $S_{\sigma, F}$ and $C_{\sigma, F}$.  The next corollary summarizes results we obtain for various classes of toric vector bundles. 

\begin{corollary}\label{cor-positivity}
    Let $\E$ be a toric vector bundle with $\RR(\P\E)$ generated in $\Sym$-degree 1, then:\\

    \begin{enumerate}
        \item If $\E$ is monomial, then any Nef class is basepoint-free,
        \item If $\E$ is uniform or sparse over $\P^{\nn}$, then any Nef class is basepoint-free, and any ample class is very ample. \\
    \end{enumerate}
    
\end{corollary}

For the definitions of monomial bundles, sparse bundles, and uniform bundles see Section \ref{sec-background}.  The class of sparse bundles includes all rank $2$ vector bundles and tangent bundles of smooth toric varieties.  Uniform bundles are those bundles given by $(L, D)$ where $L$ is a very general linear ideal.  

Corollary \ref{cor-positivity} implies that several large classes of bundles that we study satisfy one of Fujita's freeness and ampleness conjectures: 

\begin{conjecture}\label{conj-Fujita}
Let $X$ be smooth of dimension $n$ and let $A \in \CL(X)$ be ample, then\\

\begin{enumerate}
    \item for $m \geq n + 1$, $K_X + m A$ is globally generated,
    \item for $m \geq n + 2$, $K_X + m A$ is very ample.\\
\end{enumerate}
\end{conjecture}

\noindent

Mori's cone theorem (\cite[Theorem 1.5.33]{Lazarsfeld}) shows that if $X$ and $A$ are as in \ref{conj-Fujita}, then $K_X + mA$ is Nef if $m \geq n + 1$ and ample if $m \geq n + 2$ (see \cite[Section 10.4]{Lazarsfeld}). Therefore, to prove Fujita's freeness and ampleness conjectures it is sufficient to show that Nef implies globally generated and ample implies very ample on $\P\E$. Smooth toric varieties can be shown to satisfy Fujita's conjectures along these lines. Our strategy for establishing these conjectures for various projectivized toric vector bundles also follows this argument; however, we give an example of a rank $2$ projectivized toric vector bundle over a surface where Nef does not imply basepoint-free.  

Fujita's conjectures are true for smooth toric varieties, and the freeness conjecture has been shown for all toric varieties by Fujino \cite{Fujino}. Payne \cite{Payne-Ampleness} shows the ampleness conjecture for toric varieties with Gorenstein singularities.  For curves, both conjectures follow from Riemann-Roch, and both conjectures hold for surfaces by a result of Reider \cite{Reider}. The freeness conjecture has also been proved for smooth projective varieties up to dimension $5$ \cite{Ein-Lazarsfeld,Kawamata,Ye-Zhu}.  In the toric vector bundle case, if the rank of $\E$ is $2$, then $\P\E$ is known to be a complexity-$1$ $T$-variety and satisfies Fujita's freeness conjecture by a result of Altmann and Ilten \cite{AltmannIlten}. In that same paper, Altmann and Ilten also suggest that the Fujita conjectures could be studied for projectivized toric vector bundles. We also mention here the work of Fahrner \cite{Fahrner}, on algorithms to test if Mori dream spaces satisfy Fujita's freeness conjecture. Our treatment of $\Bpf(\P\E)$ is very much in line with Fahrner's work. 

\subsection{Fano bundles, non-Fano bundles }

In many of the cases we consider the Cox ring $\RR(\P\E)$ is presented as a complete intersection, so we are able to explicitly construct the anticanonical class $-K_{\P\E}$, along with the Nef cone via Theorem \ref{thm-main-nef}.  We focus on a large class of toric vector bundles, the \emph{Kaneyama bundles} in Section \ref{sec-Fano}, and show that these bundles are almost never Fano. 

\begin{theorem}\label{thm-mainFano}
    Let $\E$ be a Kaneyama toric vector bundle over $X(\Sigma)$ with diagram the diagonal matrix with entries $a_1 \leq \cdots \leq a_n$, then:\\ 

    \begin{enumerate}
        \item If $-K_{\P\E}$ is Nef, then $X(\Sigma) = \P^{n_1}\times \cdots \times \P^{n_\ell}$. Moreover, if $\ell > 1$ then $-K_{\P\E}$ is Nef if and only if $a_i = 1$ or $0$ for all $i$. If $\ell =1$ then $-K_{\P\E}$ is Nef if and only if $\sum_{i=2}^n (a_i - a_1) \leq (n - a_1)$,
        \item $-K_{\P\E}$ is ample if and only if $X(\Sigma) = \P^{n-1}$ and $\sum_{i=2}^n (a_i - a_1) < (n - a_1)$.\\
        
    \end{enumerate}
\end{theorem}

Curiously, the class of Kaneyama bundles includes the tangent bundles of smooth toric varieties.  As a corollary to Theorem \ref{thm-mainFano} we show that the only smooth toric variety with an ample tangent bundle is projective space (Remark \ref{rem-Mori}).  This is a toric special case of a theorem due to Mori in the general case \cite{Mori}, proving a conjecture of Hartshorne \cite{Hartshorne-ample}.  See \cite{Wu} for a different combinatorial proof that the only smooth toric variety with ample tangent bundle is projective space. We also show that the only smooth toric variety with a Nef tangent bundle is a product of projective spaces. 

\bigskip

\noindent{\bf Acknowledgements:} We thank Ian Cavey for many useful discussions about Examples \ref{ex-Hilb2} and \ref{ex-Hilb2again}. We thank Kiumars Kaveh for many useful discussions about toric vector bundle. 

%\subsection{Notation}

%\begin{enumerate}
 %   \item[] $N$, $M$, lattices with $M = \Hom(N, \Z)$.\\
 %   \item[] $d$, the rank of $N$.\\
 %   \item[] $T$, the algebraic torus with cocharacter group $N$.\\
 %   \item[] $\sigma$, a rational polyhedral cone in $N\otimes \Q$.\\
 %   \item[] $\Sigma$, a rational polyhedral fan in $N\otimes \Q$.\\
 %   \item[] $X(\Sigma)$, the normal toric variety associated to $\Sigma$.\\
 %   %\item[] $n= |\Sigma(1)|$, the number of rays in $\Sigma$.\\
 %   \item[] $\E$, a toric vector bundle.\\
 %   \item[] $\P\E$, the projectivization of a toric vector bundle.\\
 %   \item[] $\RR(X)$, the Cox ring of a normal, projective variety $X$.\\
 %   \item[] $\CL(X)$, the divisor class group of $X$.\\
 %   \item[] $\Pic(X)$, the Picard group of $X$.\\
 %   \item[] $\PsEff(X)$, the pseudo-effective cone of $X$.\\
 %   \item[] $\Eff(X)$, the effective monoid of $X$.\\
 %   \item[] $\Nef(X)$, the Nef cone of $X$.\\
 %   \item[] $\Bpf(X)$, the basepoint-free monoid of $X$.\\
 %   \item[] $\GF(I)$, the \Gr\"obner fan of an ideal $I$.\\
 %   \item[] $\Trop(I)$, the tropical variety of $I$.\\
 %   \item[] $\MM(L)$, the matroid associated to a linear ideal $L$.\\
 %   \item[] $\In_w(I)$, the initial ideal of $I$ with respect to a term order $w$.\\
%end{enumerate}

\section{Background on Toric Vector Bundles}\label{sec-background}

In this section we recall the classification of toric vector bundles in \cite{Kaveh-Manon-tvb}, as well as the description of the Cox rings of projectivized toric vector bundles.  For background on toric varieties we direct the reader to the book of Cox, Little, and Schenck \cite{CLS}.  

Let $\Sigma$ be a rational polyhedral fan in $N\otimes \Q$, and let $X(\Sigma)$ denote the corresponding normal toric variety over $\C$.  We always assume that $\Sigma$ is a smooth fan.  We let $\Sigma(k)$ be the set of faces of dimension $k$ in $\Sigma$.  The most important of these sets for our purposes is the set of rays $\Sigma(1)$. We let $n = |\Sigma(1)|$. For a maximal face $\sigma \in \Sigma(d)$, $q_\sigma \in X(\Sigma)$ denotes the corresponding $T$-fixed point of $X(\Sigma)$. 

Let $\Pic_{T}(X(\Sigma))$ denote the $T$-linearized line bundles on $X(\Sigma)$.  It is well-known that an element of $\Pic_{T}(X(\Sigma))$ corresponds to an integral piecewise-linear function $\psi: |\Sigma| \to \R$, and in turn that $\psi$ is determined by its values $(r_1, \ldots, r_n)$ on the ray generators $u_1, \ldots u_n$ of the fan $\Sigma$.  Accordingly, $\Pic_{T}(X(\Sigma))$ is isomorphic to $\Z^n$. Recall that two elements $(r_1, \ldots, r_n), (s_1, \ldots, s_n)$ are $T$-linearizations of the same underlying line bundle if and only if their difference $(r_1-s_1, \ldots, r_n - s_n)$ is of the form $(\langle u_1, m\rangle, \ldots, \langle u_n, m\rangle)$ for some character $m \in M$.  This leads to the exact sequence:

\[0 \to M \to \Z^n \to \Pic(X(\Sigma)) \to 0.\]\\

\noindent
The toric variety $X(\Sigma)$ is smooth, so we also have that $\CL(X(\Sigma)) \cong \Pic(X(\Sigma))$. Under this isomorphism, the $T$-linearized line bundle $\O(\br)$ corresponds to the $T$-divisor class $\sum r_i \be_i$, where $\be_i \in \CL(X(\Sigma))$ is the divisor class of the prime divisor corresponding to the ray $\rho_i \in \Sigma(1)$. 

\subsection{A tropical characterization of toric vector bundles}

The construction in \cite[Section 4]{Kaveh-Manon-tvb} associates a toric vector bundle $\E$ of rank $r$ over $X(\Sigma)$ to a linear ideal $L \subset \C[y_1, \ldots, y_m]$ and an $n \times m$ integral matrix $D$.  The matrix $D$ is called the \emph{diagram} of $\E$. We let $D_j$ denote the $j$-th column of $D$.  The rows of $D$ are in bijection with the rays of $\Sigma$; we let $w_i$ denote the row associated to $\rho_i \in \Sigma(1)$. These data must satisfy compatibility conditions:\\

\begin{enumerate}
    \item The polynomial ring $\C[y_1, \ldots, y_m]/L$ has dimension $r$,
    \item Each $w_i$ is a point in the tropicalized linear space $\Trop(L) \subset \Q^m$,
    \item For any face $\sigma \in \Sigma$ the rows $w_i$ for $\rho_i \in \sigma(1)$ must all belong to a common \emph{apartment} $A_\B \subset \Trop(L)$.\\     
\end{enumerate}

Apartments $A_\B \subset \Trop(L)$ are distinguished polyhedral subcomplexes of $\Trop(L)$.  In what follows let $\MM(L)$ denote the matroid defined by $L$ on the set $\{y_1, \ldots, y_m\}$.  In particular, a subset $I \subset \{y_1, \ldots, y_m\}$ is declared independent when no elements of $L$ are supported on $I$. There is an apartment $A_\B$ for each basis $\B \subset \MM(L)$, in particular, $A_\B$ is the set of $(v_1, \ldots, v_m) \in \Trop(L)$, where $v_j$ is equal to the minimum of the $v_k$ where $y_k$ appears in the $\B$-expression for $y_j$. It is straightforward to show that the apartments cover $\Trop(L)$ and that each $A_\B$ is piecewise-linear isomorphic to $\Q^r$. 

Let $\GF(L)$ denote the \emph{Gr\"obner fan} of $L$, see \cite{Sturmfels,Cox-Little-OShea}. As $L$ is linear, the fan $\GF(L)$ coincides with the inner normal fan of the \emph{matroid polytope} of the matroid $\MM(L)$. In particular, $\GF(L)$ is a complete fan whose maximal faces correspond to bases in $\MM(L)$.  The tropical variety $\Trop(L)$ coincides with the \emph{Bergman fan} of $\MM(L)$, and is naturally identified with the support of a subfan of $\GF(L)$. The apartment $A_\B \subset \Trop(L)$ is precisely the intersection of $\Trop(L)$ with the maximal face of $\GF(L)$ associated to $\B$.  

Let $\sigma \in \Sigma(d)$ be a maximal face, and let $\In_\sigma(L)$ be the initial ideal corresponding to the minimal face of the Gr\"obner fan of $L$ containing the rows of $D$ which come from $\sigma(1)$. We let $\P\E_\sigma$ denote the fiber of $\E$ over the $T$-fixed point $q_\sigma$. We have $\P\E_\sigma = \Proj(\C[y_1, \ldots, y_m]/\In_\sigma(L))$.  In particular, any standard basis $y_{i_1}, \ldots, y_{i_r}$ of $\In_\sigma(L)$ forms a coordinate system for $\P\E_\sigma$.

The set of diagrams $D \in \Q^{m \times n}$ which satisfy $(2)$ and $(3)$ above is denoted $\Delta(\Sigma, L)$. Observe that $\Delta(\Sigma, L)$ is a rational polyhedral fan. By \cite[Theorem 1.4]{Kaveh-Manon-tvb}, any integral point $D \in \Delta(\Sigma, L)$ corresponds to a toric vector bundle over $X(\Sigma)$, and moreover any toric vector bundle can be obtained from some pair $(L, D)$.  However, a fixed toric vector bundle can be realized by many distinct $(L, D)$. 

Given $\br = (r_1, \ldots, r_n) \in \Q^n$ and $D \in \Delta(\Sigma,L)$ we obtain a new point $D\hat\otimes \br \in \Delta(\Sigma, L)$ by adding $r_i$ to all entries in $w_i$. This operation defines an action of $\Q^n$ on $\Delta(\Sigma, L)$.  When $\br$ is integral there is a corresponding $T$-linearized line bundle $\O(\br)$ on $X(\Sigma)$, and the diagram $D\hat\otimes \br$ corresponds to the bundle $\E \otimes \O(\br)$. Tensoring with a line bundle does not change the projectivization of a vector bundle. For this reason, when we are dealing with a projectivized toric vector bundle, we can assume without loss of generality that all entries of $D$ are non-negative, with at least one $0$ entry in each row. Accordingly, we call such a diagram $D$ \emph{non-negative}, and we refer to the \emph{non-negative form} of a diagram $D$.  Equivalently, we can work with the fan $\Delta(\Sigma, L)/\Q^n \subset \Q^{n\times m}/\Q^n$.  It is straightforward to show that if $L \subseteq L'$ then $\Delta(\Sigma, L') \subseteq \Delta(\Sigma, L)$. 

\begin{lemma}\label{lem-bundlemap}
Let $L \subseteq L' \subseteq \C[y_1, \ldots, y_m]$ be linear ideals, and suppose that $D \in \Delta(\Sigma, L') \subset \Delta(\Sigma, L)$, and let $\E$, $\E'$ denote the the toric vector bundles on $X(\Sigma)$ corresponding to $L, L'$ respectively, then there is a surjection of toric vector bundles $\E \to \E'$.
\end{lemma}

\begin{proof}
    The existence of the map $\E \to \E'$ follows from \cite[Section 3]{Kaveh-Manon-tvb}. It can be checked that this map is surjection by passing to the maps $\E_\sigma \to \E_\sigma'$ defined on the torus fixed point fibers and noting that $\In_\sigma L \subseteq \In_\sigma L'$ by assumption.
\end{proof}

We finish this section by explaining the connection between the pair $(L, D)$ defining a toric vector bundle $\E$ and the Klyachko data \cite{Klyachko} of $\E$. Let $E$ denote the fiber of $\E$ over the identity of the torus $T$, viewed as an open subvariety of $X(\Sigma)$, for a ray $\rho_i \in \Sigma(1)$ let $u_i$ denote the integral ray generator.  By Klyachko's result, a toric vector bundle $\E$ corresponds to the following data:\\

\begin{enumerate}
\item for every ray $\rho \in \Sigma(1)$ a decreasing integral filtration: $$E \supseteq \cdots \supseteq F^\rho_r \supseteq F^\rho_{r+1} \supseteq \cdots \supseteq 0.$$
\item for every maximal face $\sigma \in \Sigma$ a splitting labelled by characters: $E = \bigoplus_{m \in M} L_m$,
\item such that if $\rho_i \subset \sigma$ we have $F^{\rho_i}_r = \bigoplus_{\langle u_i, m \rangle \geq r} L_m$.\\
\end{enumerate}

\noindent
From the perspective of the data $(L, D)$, $E$ is the first graded component of the quotient ring $\C[y_1, \ldots, y_m]/L$, and each space $F^{\rho_i}_r$ is the flat of the matroid $\MM(L)$ defined by those $y_j$ with $D_{ij} \geq r$. The existence of the splitting in $(2)$ above is then the condition that requires the rows $w_i$ for $i \in \sigma(1)$ to live in a common apartment for every face $\sigma \in \Sigma$. 

\subsection{The Cox ring of a projectivized toric vector bundle}

The Cox ring of a toric vector bundle is described in \cite{HMP}, \cite{Gonzalez-rank2}, \cite{GHPS}, \cite{Nodland}, and \cite{Kaveh-Manon-NOK}. We assume that $X(\Sigma)$ is smooth so that $\P\E$ is smooth, and $\CL(\P\E) \cong \Pic(\P\E) \cong \Pic(X(\Sigma))\times \Z$. The line bundles in the $\Z$ component of this splitting correspond to the Serre sheaves $\O_\E(\ell)$, $\ell \in \Z$ which emerge from regarding $\P\E$ as $\Proj_{X(\Sigma)}(\bigoplus \Sym^\ell(\E))$.  The $\Pic(X(\Sigma))$ component of this splitting comes from pullbacks $\pi^*\L$ of line bundles on the base: $\L \in \Pic(X(\Sigma))$.  Then, any line bundle on $\P\E$ is of the form $\pi^*\L \otimes \O_\E(\ell)$, and by the projection formula (see \cite[Chapter 2, Exercise 7.9]{hartshorne}, and \cite{HMP}) we have $H^0(\P\E, \pi^*\L \otimes \O_\E(\ell)) \cong H^0(X(\Sigma), \L\otimes \Sym^\ell(\E))$. This leads to the following expression for the Cox ring $\RR(\P\E)$ of $\P\E$:\\

\[\RR(\P\E) = \bigoplus_{\L \in \Pic(X(\Sigma)), \ell \in \Z} H^0(X(\Sigma), \L\otimes \Sym^\ell(\E)). \]\\

\noindent
We obtain a finer grading of $\RR(\P\E)$ by the $T$-section spaces:\\

\[\RR(\P\E) = \bigoplus_{\br \in \Z^n, \ell \in \Z} H_{T}^0(X(\Sigma), \O(\br) \otimes \Sym^\ell(\E)).\]\\

\noindent
The second expression has additional useful features. The section spaces $H_{T}^0(X(\Sigma), \O(\br) \otimes \Sym^\ell(\E))$ admit a combinatorial description by their \emph{parliament of polytopes} \cite{DJS}. Moreover, these spaces are the components of a multi parameter filtration of the polynomial ring $\C[y_1, \ldots, y_]/L \cong \Sym(E)$ defined by the data of the diagram $D$.  For a polynomial $p \in \C[y_1, \ldots, y_m]$ and $w \in \Z^m$, the $w$-weight of $p(\bar{y}) = \sum C_\alpha y^\alpha$ is defined to be $\min\{\langle w, \alpha \rangle \mid C_\alpha \neq 0\}$. Now let $(L, D)$ define the toric vector bundle $\E$.  For a row $w$ in $D$ we define $F^{w}_r \subset \C[y_1, \ldots, y_m]/L \cong \Sym(E)$ to be the set of elements $f$ such that there exists a polynomial $p \in \C[y_1, \ldots, y_m]$ which maps to $f$ in the quotient, with $w$-weight greater than or equal to $r$. Let $F^{w}_r(\ell)$ denote the intersection of this space with the elements of $\C[y_1, \ldots, y_m]/L$ of homogeneous degree $\ell$. The following is shown in \cite[Section 5]{Kaveh-Manon-tvb}.

\begin{proposition}
For $\mathbf{r} \in \Z^n$ and $\ell \in \Z$ we have:\\

\[H_{T}^0(X(\Sigma), \O(-\br) \otimes \Sym^\ell(\E)) = \bigcap_{1 \leq i \leq n} F^{w_i}_{r_i}(\ell).\]\\

\end{proposition}

\noindent
If the requirement that $\Sigma$ be complete is relaxed, then the apartment conditions defining $\Delta(\Sigma, L)$ can be essentially omitted by deleting all higher codimension torus orbits. Alternatively, if $D$ is any point on $\Trop(L)^n$ and $\Sigma$ defines a toric surface, one can always enlarge the set of generators of $\Sym(E)$ to find $L' \subset \C[y_1, \ldots, y_{m'}]$ so that the associated extended diagram $D' \in \Delta(\Sigma, L')$ defines the same set of filtrations on $\Sym(E)$.  As a consequence, we see that the Cox rings of projectivized toric vector bundles are precisely the class of \emph{multi-Rees algebras} defined by configurations of integral points on a tropicalized linear space $\Trop(L)$. 

The second author and Kaveh define an algorithm \cite[Algorithm 5.6]{Kaveh-Manon-tvb} which constructs a finite generating set of $\RR(\P\E)$, provided one exists.  This is an adaptation of \cite[Algorithm 1.18]{Kaveh-Manon-NOK}, which constructs a finite Khovanskii basis for an algebra quasivaluation.  This algorithm can be implemented in symbolic algebra software like Macaulay2 \cite{m2}. 

\subsection{Presentations of the Cox ring $\RR(\P\E)$ and $\Sym$-degree 1}

General presentations of $\RR(\P\E)$ are studied in \cite{Kaveh-Manon-tvb}.  In this paper we are interested in projectivized toric vector bundles whose Cox rings are generated in minimal degree. 

\begin{definition}
We say $\RR(\P\E)$ is \emph{generated in $\Sym$-degree $1$} if the components $H^0(X(\Sigma), \L)$, $H^0(X(\Sigma), \L\otimes \E) \subset \RR(\P\E)$ for $\L \in \Pic(X(\Sigma))$ suffice to generate $\RR(\P\E)$. 
\end{definition}

Let $X_i \in \RR(\P\E)$ for $1 \leq i \leq n$ denote the section of the pullback of the prime divisor in $X(\Sigma)$ corresponding to the $i$-th ray of $\Sigma$. For each generator $y_j \in \C[y_1, \ldots, y_m]$ there is a corresponding section $Y_j \in F_{D_{1j}}^{w_1} \cap \cdots \cap F_{D_{nj}}^{w_n}$. The Cox ring $\RR(\P\E)$ can be realized as a subring of the partially Laurent polynomial ring $\Sym(E)[t_1^\pm, \ldots, t_n^\pm]$.  The generators $X_i$ and $Y_j$ are mapped into $\Sym(E)[t_1^\pm, \ldots, t_n^\pm]$ as follows:\\

\[X_i \to t_i^{-1},\]

\[Y_j \to y_jt^{D_j},\]\\

\noindent
where $t^{D_j} = t_1^{D_{1j}}\cdots t_n^{D_{nj}}$.  We let $\I$ denote the kernel of the map $\Phi: \C[X_1, \ldots, X_n, Y_1, \ldots, Y_m] \to \Sym(E)[t_1^\pm, \ldots, t_n^\pm]$ defined by these expressions. The following provides a tropical and algebraic criterion for $\RR(\P\E)$ to be generated in $\Sym$-degree $1$, it is a special case of \cite[Theorem 1.6]{Kaveh-Manon-tvb}.

\begin{proposition}\label{prop-sd1}
Let $\E$ be a toric vector bundle over $X(\Sigma)$, then each row $w_i$ of $D$ lifts to a point $\hat{w}_i \in \Trop(\I)$.  Moreover, $\RR(\P\E)$ is generated in $\Sym$-degree $1$ if and only if $\In_{\hat{w}_i}(\I)$ is a prime ideal for all $1 \leq i \leq n$.  
\end{proposition}

Now we fix some notation for special classes in $\CL(\P\E) \cong \Pic(\P\E)$. Recall that $\be_i \in \CL(X(\Sigma) \cong \Pic(X(\Sigma))$ denotes the class of the toric divisor of $X(\Sigma)$ corresponding to the $i$-th ray. The $j$-th column of the diagram $D$ defines the class $\bd_j = \sum_{i =1}^n D_{ij}\be_i \in \CL(X(\Sigma))$.  Under the $\CL(\P\E)$ grading of $\RR(\P\E)$ we have:\\

\[\deg(X_i) = (-\be_i, 0),\]

\[\deg(Y_j) = (\bd_j, 1).\]\\

We let $\PsEff(\P\E) \subset \CL(\P\E)\otimes \Q$ denote the pseudo-effective cone of $\P\E$, and $\Eff(\P\E) \subset \CL(\P\E)$ denote the monoid of effective classes. 

\begin{proposition}
Let $\E$ be a toric vector bundle over $X(\Sigma)$ with $\RR(\P\E)$ generated in $\Sym$-degree $1$, then:\\

\[\PsEff(\P\E) = -\PsEff(X(\Sigma))\times \{0\} + \sum_{j =1}^m \Q_{\geq 0}(\bd_j, 1),\]

\[\Eff(\P\E) = -\Eff(X(\Sigma))\times \{0\} + \sum_{j =1}^m \Z_{\geq 0}(\bd_j, 1).\]\\

\end{proposition}

Every bundle $\E$ with $\RR(\P\E)$ generated in $\Sym$-degree $1$ has a distinguished embedding into a split toric vector bundle.  Let $\E$ be defined by the pair $(L, D)$ with $D$ non-negative. We let $\V_D = \bigoplus_{j = 1}^m \O(D_j)$, where $\O(D_j)$ denotes the $T$-linearized toric line bundle on $X(\Sigma)$ corresponding to $D_j \in \Z^n$.  In general, an embedding of a Mori dream space $X$ into a toric variety $Z$ is said to be a \emph{neat embedding} (\cite{ADUH-book}) if it induces a surjection on Cox rings and an isomorphism of class groups.

\begin{proposition}\label{prop-neat}
There is a surjection of toric vector bundles: $\V_D \to \E$, and a corresponding neat embedding $\P\E \to \P\V_D$.   
\end{proposition}

\begin{proof}
The surjection $\V_D \to \E$, and therefore the embedding $\P\E \to \P\V_D$ are consequences of Lemma \ref{lem-bundlemap}.  By definition, the Serre sheaf $\O_{\V_D}(1)$ pulls back to $\O_\E(1)$ under this embedding, and pullbacks of line bundles on $X(\Sigma)$ are pulled back to pullbacks of line bundles from $X(\Sigma)$. Consequently the class group $\CL(\P\V_D)$ is mapped isomorphically onto $\CL(\P\E)$. Moreover, the embedding induces the map on graded rings $\C[X_1, \ldots, X_n, Y_1, \ldots, Y_j] \to \RR(\P\E)$, where the polynomial ring can be interpretted as $\RR(\P\V_D)$.
\end{proof}

Proposition \ref{prop-neat} implies that $\Eff(\P\E) \cong \Eff(\P\V_D)$, $\PsEff(\P\E) \cong \PsEff(\P\V_D)$, and that the map $H^0(\P\V_D, \O(\alpha, \beta)) \to H^0(\P\E,\O(\alpha, \beta))$ is a surjection for all $(\alpha, \beta) \in \CL(X(\Sigma))\times \Z$.  As a consequence we can give a useful description of the section spaces of $\P\E$ in terms of certain integral polytopes. Consider the map $\deg: \Q^{n+m} \to \CL(\P\V_D)\otimes \Q$ defined by $X_i \to (-\be_i, 0)$ and $Y_j \to (\bd_j, 1)$. For $(\alpha, \beta) \in \CL(\P\V_D)\otimes \Q$ we define the following polytope:\\

\[P_{\alpha, \beta} = \{(a,b) \mid  \beta = \sum b_j, \sum b_j\bd_j + \alpha = \sum a_i\be_i\} \subset \Q_{\geq 0}^{n+m}.\]\\  

\noindent
For $(a,b) \in P_{\alpha, \beta}$, $b$ defines an integral point of the $\beta$-th Minkowski sum of the standard $m-1$ simplex, and $a \in \Q^n$ maps to $\sum b_j\beta_j + a$ under the map $\Q^n \to \CL(X(\Sigma))\otimes \Q$.

\begin{proposition}
    The integral points in $P_{\alpha, \beta}$ coincide with a basis of $H^0(\P\V_D,\O(\alpha, \beta))$.  
\end{proposition}

By considering a slight change of coordinates we show that $P_{\alpha, \beta}$ is always a \emph{generalized Cayley polytope}.  Pick a  section to the surjection $\Z^n \to \CL(X(\Sigma))$; this amounts to choosing a $T$ linearization of every divisor class on $X(\Sigma)$ in a way which is compatible with addition in $\CL(X(\Sigma))$. Such a section determines a splitting $\Z^n \cong M \oplus \CL(X(\Sigma))$.  Accordingly, each divisor class $\bd \in \CL(X(\Sigma))$ is assigned a poltyope $\Delta(\bd)$; we have $\Delta(\bd_1 + \bd_2) = \Delta(\bd_1) + \Delta(\bd_2)$, and the rational points in $\Delta(\bd)$ are precisely the image of those $a \in \Q_{\geq 0}^n$ which map to $\bd$ under the projection $\Q^n \to M\otimes \Q$. Let $\psi: \Q^{n+m} \to \Q^m \times M\otimes \Q$ be the projection map induced from this splitting.  

\begin{proposition}\label{prop-nobycayley}
Assume that $\beta_j\bd_j + \alpha \in \Eff(X(\sigma))$ for all $1 \leq j \leq m$. The polytope $P_{\alpha, \beta}$ is mapped isomorphically onto $\psi\circ P_{\alpha, \beta}$. Moreover, $\psi\circ P_{\alpha, \beta}$ is the generalized Cayley polytope obtained by placing $\Delta(\beta\bd_j + \alpha)$ above $\beta e_j \in \Q^m$.  
\end{proposition}

\begin{proof} 
First observe that $P_{\alpha, \beta}$ naturally maps onto the $\beta$-Minkowski sum of the standard $m-1$ simplex by sending $(a, b)$ to $b$. Now we map $(a, b)$ to the point $(m, \sum a_i\be_i, b) = (m,\sum b_j b + \alpha, b)$, and then to $(m, b)$.  The composition of these maps is $\psi$, $m \in \Delta(\sum b_j b + \alpha)$, and all such $m$ are obtained this way be definition. But the middle component in the interim step is determined by $\psi(a, b)$. This shows that $P_{\alpha, \beta} \cong \psi \circ P_{\alpha, \beta}$.  Finally, the $b$ fiber of $\psi\circ P_{\alpha,\beta}$ is $\Delta(\sum b_j\bd_j + \alpha) = \sum \Delta(\beta \bd_j + \alpha)$, which is the $b$ fiber the required Cayley polytope. 
\end{proof}

\subsection{CI Bundles}

For any linear form $\ell(\bar{y}) = \sum_{j = 1}^m C_jy_j \in L$ there is a corresponding relation $p(\bar{X}, \bar{Y}) = \frac{1}{X^\delta}\sum_{j =1}^m C_j X^{D_j}Y_j \in \I$, where $\delta$ denotes the least common multiple of the monomials $X^{D_j}$ with $C_j \neq 0$. In \cite{Kaveh-Manon-tvb} a combinatorial classification of those toric vector bundles which are presented as a complete intersection by relations of the form $p(\bar{X}, \bar{Y})$ is given. Fix a generating set $\ell_1, \ldots, \ell_{m-r} \in L$, and let $M$ be the $m-r \times m$ matrix of coefficients of $\ell_1, \ldots, \ell_{m-r}$. For each $A \subset [n]$ we define a new matrix $M_A$.  The entry $M_A(k,j)$ equals the $k,j$-th entry of $M$ if the rows $w_i$, $i \in A$ of $D$ have a common minimal entry at $j$, and is $0$ otherwise. We let $m_A$ be the rank of $M_A$. The following is \cite[Proposition 6.2]{Kaveh-Manon-tvb}.

\begin{theorem}\label{thm-CIcondition}
The Cox ring $\RR(\P\E)$ is presented in $\Sym$-degree $1$ by the polynomials $p_1, \ldots, p_{m-r}$ if and only if for every $A \subset [n]$ we have $r < |A| + m_A$, and for every $B \subset [n]$ and $i \in B$ we have $1 + m_{\{i\}} < |B| + m_B$. 
\end{theorem}

For $\ell = \sum C_jy_j \in L$ let $\trop(\ell)$ denote the piecewise-linear function $\MIN\{y_j \mid C_j \neq 0\}$. The subsets of $\Q^m$ where $\trop(\ell)$ is a linear function subdivide $\Q^m$ into a complete polyhedral fan $\Sigma(\ell)$.  We let $\F(\Sigma, \ell_1, \ldots, \ell_{m-r}) \subset \Trop(L)^n$ denote the fan obtained by refining $\Delta(\Sigma, L)$ with respect to $\Sigma(\ell_k)^n$ for $1 \leq k \leq m-r$. Observe that the diagrams $D$ which satisfy the inequalities of Proposition \ref{thm-CIcondition} correspond to integral points in faces of $\F(\Sigma, \ell_1, \ldots, \ell_{m-r})$. Moreover, this fan clearly only depends on the supports of the forms $\ell_1, \ldots, \ell_r$, and there are only a finite number of sets of supports of generators of $L$. It follows that we may further refine $\Delta(\Sigma, L)$ with respect to all generating sets of $L$ to get a fan $\F(\Sigma, L)$.  

\begin{theorem}
There is an integral polyhedral fan $\F(\Sigma, L) \subset \Q^{n\times m}$ such that $\RR(\P\E)$ is presented in $\Sym$-degree $1$ by polynomials $p_1, \ldots, p_{m-r}$ for a collection of linear generators $\ell_1, \ldots, \ell_{m-r} \in L$ if and only if $D$ is an integral in the relative interior of certain faces of $\F(\Sigma, L)$. 
\end{theorem}

\noindent
We call a toric vector bundle $\E$ corresponding to such a diagram $D$ a \emph{complete intersection} toric vector bundle or \emph{CI} toric vector bundle for short. For a homogenized linear relation $p_k \in \I$ we let $(\bc_k, 1) \in \CL(\P\E)$ denote the corresponding degree.

\begin{proposition}
If $\E$ is a $CI$ bundle with $\I = \langle p_1, \ldots, p_{m-r}\rangle$, then:\\

\[-K_{\P\E} = (K_\Sigma + \sum_{i =1}^m \bd_j - \sum_{k =1}^{m-r} \bc_k, r).\]\\

\end{proposition}

There are several interesting subclasses of CI toric vector bundles.  By requiring that additional conditions be satisfied by the pair $(L, D)$ we can simplify the verification of the hypothesis of Proposition \ref{thm-CIcondition}.  First we consider what happens when we take $L$ to be a very general linear ideal. 

\begin{definition}[Uniform Toric Vector Bundles]
We say a toric vector bundle $\E$ is \emph{uniform} if it is the toric vector bundle determined by a pair $(L, D)$ where $\MM(L)$ is a uniform matroid. 
\end{definition}

The Klyachko spaces of a uniform toric vector bundle $\E$ can be realized as flats of a uniform matroid $U^r_m$. The conditions of Proposition \ref{thm-CIcondition} simplify considerably when $\E$ is assumed to be uniform. 

\begin{theorem}\label{theorem:uniform CI condition}
Let $\E$ be a uniform toric vector bundle corresponding to the pair $(L, D)$, where $L$ is a very general linear ideal, then $\E$ is $CI$ if and only if for every collection $A \subset [n]$, the non-zero entries of the rows of $D$ corresponding to the indices in $A$ are contained in $r + |A| - 2$ columns. 
\end{theorem}

\noindent
Here the condition that $L$ be very general means that $L$ is the kernel of a matrix with all non-vanishing minors. The condition for a uniform toric vector bundle to be a complete intersection toric vector bundle only depends on the location of the $0$ entries in $D$.  As a consequence, the faces of the fan $\F(\Sigma, L)$ which contain the diagrams of $CI$ uniform bundles can be described as certain $n$-fold products of closed faces from the Bergman fan $\Trop(U^r_m)$. 

A uniform bundle has the most tractable linear ideal $L$. The next class of toric vector bundle has a simple diagram $D$.

\begin{definition}[Sparse Toric Vector Bundles] \label{def-sparse}
We say that a toric vector bundle $\E$ is \emph{sparse} if it determined by a pair $(L, D)$ where there non-negative form of $D$ has at most one non-zero entry in each row.  
\end{definition}

By \cite[Corollary 6.7]{Kaveh-Manon-tvb} any sparse toric vector bundle $\E$ is a complete intersection toric vector bundle, and so has $\P\E$ a Mori dream space.  This class of toric vector bundles was first studied by Gonz\'alez, Hering, Payne, and S\"u\ss \ in \cite{GHPS}, where they were described as toric vector bundles with at most one step of codimension 1 in the Klyachko filtrations. The sparse toric vector bundles include tangent bundles of smooth toric varieties, first studied in \cite{Hausen-Suss}, and all rank $2$ toric vector bundles \cite{Gonzalez-rank2}.  The projectivization of a sparse toric vector bundle is an example of an \emph{arrangement variety} \cite{Cummings-Manon}, and the projectivizations of toric vector bundles which are both uniform and sparse are \emph{general arrangement varieties} \cite{Hausen-Hische-Wrobel}. 

\subsection{New CI bundles from old CI bundles}\label{sec-extension}

We describe a number of extension theorems, allowing us to form infinitely many new CI bundles of the same type. We describe these new bundles by adding columns to $D$. We suppose an $m-r \times m$ matrix $M$ has been given such that $L = \ker(M)$.  In particular, we take the rows of $M$ to be a minimal generating set of $L$. In what follows we consider the extension $(D', L')$ of a pair $(D, L)$, where $D' = [D \mid U]$ and $L' = ker(M')$ for $M' = [M \mid X]$. We let $\E$ be the toric vector bundle corresponding to the pair $(D, L)$ and let $\E'$ be the toric vector bundle corresponding to the pair $(D', L')$. For the sake of simplicity we assume that $D$ and $D'$ are non-negative. 

\begin{proposition}
    Let $\E$ be a CI toric vector bundle, and suppose that every entry of the $i$-th row of $U$ is larger than every entry in the $i$-th row of $D$, then $\E'$ is CI. 
\end{proposition}

\begin{proof}
    Let $p_1', \ldots, p_{m-r}'$ denote the homogenizations of the rows of $M'$ and $p_1, \ldots, p_{m-r}$ denote the homogenizations of the rows of $M$. For any row of $D'$, the initial forms of the $p_k'$ agree with those of the $p_k$.  It follows that the initial ideal $\In_{\hat{w}_i}(\I)$ is prime and generated by these initial forms, the theorem then follows from Proposition \ref{prop-sd1}
\end{proof}

When $\E$ is a uniform bundle it is more straightforward to find extensions. 

\begin{proposition}
    Let $\E$ be a uniform CI toric vector bundle defined by $M$ a very general matrix, and suppose that $M'$ defined by an $X$ which also makes $M'$ very general, then $\E'$ is a uniform CI bundle for any non-negative $D'$ extending $D$. 
\end{proposition}

\begin{proof}
    The diagram $D$ must satisfy the conditions of Theorem \ref{theorem:uniform CI condition}.  But it is then immediate that $D'$ also satisfies these conditions. 
\end{proof}

\begin{example}
Let $\E := \T\P^2$, the tangent bundle of $\P^2$. Then the corresponding $(L,D)$ can be taken to be $L~=~\langle y_0+y_1+y_2 \rangle \subset \C[y_0,y_1,y_2]$ and \\

$$
D = \begin{bmatrix}
    1 & 0 & 0\\
    0 & 1 & 0 \\
    0 & 0 & 1
\end{bmatrix}.
$$ \\

\noindent We observe that $\E$ is a uniform, sparse, and CI bundle. The bundle $\E$ can be extended to $\E'$ with the associated linear ideal $L' = \langle y_0 + y_1+y_2+y_3 \rangle \subset \C[y_0,y_1,y_2,y_3]$ and diagram: \\

$$
D' = \begin{bmatrix}
    1 & 0 & 0 & 1\\
    0 & 1 & 0 & 1\\
    0 & 0 & 1 & 1
\end{bmatrix}.
$$ \\

\noindent We then have the Cox ring:\\

\[\RR(\P\E') = \C[x_0, x_1, x_2, Y_0, Y_1, Y_2, Y_3]/\langle x_0Y_0 + x_1Y_1 + x_2Y_2 + x_0x_1x_2Y_3\rangle.\]\\

\noindent
This is recognizable as the Cox ring of the blow up of $3$ points on $\P^2$ which lie on a rational curve. 

\end{example}

\section{Newton-Okounkov bodies}\label{sec-noby}

In this section we use tropical methods to compute Newton-Okounkov bodies of  projectivized toric vector bundles.  First we compute the \emph{global} Newton-Okounkov body of the Cox ring, and then specialize to find Newton-Okounkov bodies of divisors.  The latter case includes the divisor $\O_\E(1)$, where the Newton-Okounkov body can be rightly thought of as the Newton-Okounkov body of $\E$ itself. 

For an introduction to Newton-Okounkov bodies see \cite{Lazarsfeld-Mustata, Kaveh-Khovanskii}.  In general, one starts with a positively graded ring $R$ equipped with a full-rank discrete valuation $\v$. The Newton-Okounkov body, a closed, compact, convex body $\Delta(R, \v)$ is assigned to this data.  When $R$ is the section ring of a divisor on a projective varieties, various properties of the associated linear series can be read off $\Delta(R, \v)$.  One can also consider a valuation on the Cox ring of a projective variety which is appropriately homogeneous with respect to the class-group grading, leading to information about how the Newton-Okounkov bodies of all divisors fit together.   

In \cite{Kaveh-Manon-NOK} it is shown that Newton-Okounkov bodies can be computed directly from a suitable tropical variety of the ring $R$.  We take this approach toward compute global Newton-Okounkov bodies of projectivized CI toric vector bundles $\P\E$. Accordingly, we begin with a description of the tropical variety of $\RR(\P\E)$ obtained from its presentation by the generators $x_1, \ldots, x_n, Y_1, \ldots, Y_m$.

\subsection{Tropicalization of the Cox Ring}

It is observed in \cite[Section 5]{Kaveh-Manon-tvb} that $\RR(\P\E)$ can be realized as a subalgebra of the Laurent polynomial ring $\Sym(E)[t_1^\pm, \ldots, t_n^\pm]$ by inverting $X_1, \ldots, X_n$.  Here $\Sym(E)$ is isomorphic to the quotient algebra $\C[y_1, \ldots, y_m]/L$.  Moreover, $\Phi$ fits into a commuting square:\\

$$
\begin{tikzcd}
\C[\bar{X}, \bar{Y}] \arrow[r, "\tilde{\Phi}"] \arrow[d]
& \C[y_1, \ldots, y_m, \bar{t}^\pm] \arrow[d] \\
\RR(\P\E) \arrow[r, "\Phi"]
& \Sym(E)[\bar{t}^\pm]
\end{tikzcd}
$$\\

\noindent
where $\tilde{\Phi}(X_i) = t_i^{-1}$ and $\tilde{\Phi}(Y_j) = y_jt^{D_j}$, and $D_j$ denotes the $j-$the column of the diagram $D$. Recall that $\I$ denotes the ideal in $\C[\bar{X}, \bar{Y}]$ which presents $\RR(\P\E)$. The following proposition summarizes the tropical consequences of this observation. 

\begin{proposition}\label{prop-tropstructure}
The map $\phi: \Q^{m+n} \to \Q^{m+n}$ given by\\

\[\phi(v_1, \ldots, v_m, m_1, \ldots, m_n) = (\ldots, v_j + \langle \bar{m}, D_j\rangle, \ldots, -m_1, \ldots, -m_n)\]\\

\noindent
induces a bijection between $\Trop(\I)$ and $\Trop(L)\times \Q^n$. 
 In particular, the projection map $\pi: \Q^m\times \Q^n \to \Q^m$ maps $\Trop(\I)$ onto $\Trop(L)$, and there is section $s: \Trop(L) \to \Trop(\I)$ of $\phi$ given by\\

\[s(v_1, \ldots, v_m) = (v_1, \ldots, v_m, 0, \ldots, 0).\]\\
\end{proposition}

The rows $w_i$ of the diagram $D$ are themselves points in $\Trop(L)$, and define points $s(w_i) = \hat{w}_i \in \Trop(\I)$ by the map above.  These are the lifts in Proposition \ref{prop-sd1} which give the criterion for $\RR(\P\E)$ to be generated in $\Sym$-degree $1$. Prime points on tropical varieties come with an associated weight valuation by \cite[Section 4]{Kaveh-Manon-NOK}.  We let $\v_{\rho_i}: \RR(\P\E)\setminus\{0\} \to \Z$ be the weight valuation of $\hat{w}_i$.

\subsection{The global Newton-Okounkov body of $\P\E$}

We compute the global Newton-Okounkov body of $\P\E$ by following the procedure described in \cite{Kaveh-Manon-NOK}:\\

\begin{enumerate}
    \item Find a dimension $r+n$ prime cone $C \subset \Trop(\I)$ (with respect to the fan structure induced from the Gr\"obner fan of $\I$).
    \item Select linearly independent integral points $\bu = \{u_1, \ldots, u_{r+n}\} \subset C$.
    \item Compute $\Delta_{C, \bu} = M \circ \Q_{\geq 0}^{m + n}$, where $M = [u_1, \ldots, u_{r+n}]^t$. \\
\end{enumerate}

The matrix $M = [u_1, \ldots, u_{r+n}]^t$ can be used to define a rank $r+n$ discrete valuation $\v_M$ on $\RR(\P\E) = \C[\bar{X},\bar{Y}]/\I$ with value semigroup $S(\RR(\P\E), \v_M) = M \circ \Z_{\geq 0}^{m+n}$ \cite[Proposition 4.2]{Kaveh-Manon-NOK}. A different set of points $\bu' = \{u_1', \ldots, u_{r+n}'\}$ results in an isomorphic value semigroup, and linearly isomorphic convex body $\Delta_{C, \bu'}$.  Accordingly,  we suppress the choice $\bu$ in the notation, however, a careful choice of $\bu$ leads to clearer results.  

For any prime cone $C$ we may choose the preimages of the $n$ basis vectors of $0\times \Q^n \subset \Trop(L)\times \Q^n$ in $\Trop(\I)$ as the first $n$ rows of our matrix $M$. These are the vectors $R_i = (D_{i1}, \ldots, D_{im},0,\ldots,-1,\ldots,0)$, so the matrix formed by these rows is $[D, -I]$, where $D$ is the diagram of $\E$.

It remains to append $r$ more points from a particular $C \subset \Trop(\I)$.  To simplify matters, we always choose these elements in the image of the section $s: \Trop(L) \to \Trop(\I)$. The rows $R_i$ $1 \leq i \leq n$ account for the grading of $\RR(\P\E)$ by $\CL(X(\Sigma)) \cong \Pic(X(\Sigma))$ and the action of the torus $T$. We can now always pick $R_{n+1} = (1, \ldots, 1, 0, \ldots, 0)$ to account for the grading by $\{0\}\times \Z \subset \Pic(X(\Sigma)) \times \Z \cong \Pic(\P\E)$. This leaves us with the task of selecting $r-1$ more vectors of the form $(w_i, 0) \in C$.  This can be done in practice by computing a Gr\"obner basis of $\I$ with respect to the weight $(w, 0)$; this allows one to compute extremal rays of $C$ of the form $(w_i, 0)$.  The situation is best when $C$ is itself the preimage of a maximal face $K \subset \Trop(L)$. 

The tropical variety $\Trop(L)$ is the Bergman fan of the matroid $\MM(L)$.  Maximal faces of Bergman fans coincide with maximal flags of flats: $F_1 \supset F_2 \supset \cdots \supset F_r$. If $K$ is the face corresponding to the flag $F_1, \ldots, F_r$, then the indicator vectors $\be_{F_k}$ form the extremal rays of $K$. In particular $\be_{F_1}$ is the indicator vector of the whole matroid: the all $1$'s vector. We have already chosen this vector for $M$, we let the remaining choices be the remaining $r-1$ indicator vectors of the flats in the flag corresponding to $K$ in decreasing order of size.  We let $E_K$ denote the matrix with rows equal to $\be_{F_i}$.  The following proposition collects the conclusions of this discussion. 

\begin{proposition}\label{prop-preimagecone}
Let $C \subset \Trop(\I)$ be a prime cone of the form $C = \phi^{-1}K$ for a maximal face $K \subset \Trop(L)$.  Let $M$ be the matrix with rows $R_1, \ldots, R_{n+1}$, along with rows $R_{n + 1 + k} = (\be_{F_k}, 0)$ $1 \leq k \leq r-1$ for the flats $F_k$ associated to the face $K$:\\

\[M = \begin{bmatrix} D & -I \\ E_K & 0 \\  \end{bmatrix}. \]\\

\noindent
The global value semigroup $S(\RR(\P\E), \v_M)$ of $\P\E$ associated to $C$ is $M\Z_{\geq 0}^{n + m} \subset \Z^{n + r}$, and the global Newton-Okounkov body $\Delta_C$ is $M\circ \Q_{\geq 0}^{n + m} \subset \Q^{n + r}$.
\end{proposition}

For a general dimension $r+n$ prime cone $C$ more work must be done to populate the final $r-1$ rows of $M$.  In the case that $\RR(\P\E)$ is generated in $\Sym$-degree $1$, the descriptions of $S(\RR(\P\E), \v_M)$ and $\Delta_C$ are identical to those in Proposition \ref{prop-preimagecone}. We give several sufficient conditions for the assumptions of Proposition \ref{prop-preimagecone} to hold. 

\begin{proposition}\label{prop-NOKgeneralrow}
Suppose that $\E$ is a CI bundle, and that a row $w$ of the diagram $D$ is contained in the relative interior of a maximal face $K \subset \Trop(L)$, then $\phi^{-1}K = C \subset \Trop(\I)$ is a prime cone of $\I$.   
\end{proposition}

\begin{proof}
By \cite[Corollary 6.7]{Kaveh-Manon-tvb}, the initial ideal $\In_{s(w)}(\I)$ is the ideal $\langle \ldots, \In_{s(w)}(p_k),\ldots \rangle$ generated by the initial forms of the $p_k$, and is prime. Moreover, by assumption each $\In_{s(w)}(p_k)$ is binomial, so $C$ must be a maximal face of $\Trop(\I)$. 
\end{proof}

The case covered in Proposition \ref{prop-NOKgeneralrow} is particularly nice because the indicator vectors $\be_{F_k}$ of the flag corresponding to $K$ can be read off the ``jumps" in the row vector $w$. Any uniform CI toric vector bundle whose diagram contains a row with distinct non-zero entries satisfies Proposition \ref{prop-NOKgeneralrow}.  Next we see that the any sparse bundle satisfies the hypotheses of Lemma \ref{prop-preimagecone}.

\begin{proposition}\label{prop-sparsewp}
Suppose that $\E$ is a sparse toric vector bundle, then every point of $\Trop(\I)$ is a prime point. 
\end{proposition}

\begin{proof}
A consequence of the condition in Theorem \ref{thm-CIcondition} is that the ideal $\I$ is prime. If $D$ is the diagram of a sparse toric vector bundle, the conditions of Theorem \ref{thm-CIcondition} is satisfied for any linear ideal not containing a monomial. For any $s(w) \in \Trop(\I)$, the complete intersection ideal induced from $\In_w(L)$ is easily shown to be contained in $\In_{s(w)}(\I)$. The former ideal is prime, and both ideals have the same height, so we must conclude that the two ideals coincide. 
\end{proof}

Proposition \ref{prop-sparsewp} shows that the ideal $\I$ which presents the Cox ring $\RR(\P\E)$ is \emph{well-poised}, \cite{Ilten-Manon,Cummings-Manon}. This can also be seen as a consequence of the fact that the projectivization $\P\E$ of a sparse toric vector bundle is an \emph{arrangement variety}, \cite{Cummings-Manon}.

\subsection{Newton-Okounkov bodies of divisors}

Now we describe the Newton-Okounkov bodies of the section rings of the divisors on $\P\E$. For $(\alpha, \beta) \in \CL(\P\E)$ we recall the polytope $P_{\alpha, \beta}$ whose lattice points define a basis of the corresponding section space over $\P\V_D$.

\begin{proposition}\label{prop-NOKbodies}
Let $\E$ be a toric vector bundle which satisfies the hypotheses of Proposition \ref{prop-preimagecone}, let $(\alpha, \beta) \in \CL(\P\E)$, and let $C \subset \Trop(\I)$ be a prime cone with matrix $M$ as above, then the Newton-Okounkov body $\Delta_C(\alpha, \beta)$ is $M\circ P_{\alpha, \beta} \subset \Q^{n+r}$. 
\end{proposition}

\begin{proof}
As constructed, $M\circ P_{\alpha, \beta}$ is the fiber above $(\alpha, \beta) \in \CL(\P\E)$ under the projection $\Delta_C \to \Eff(\P\E)$.
\end{proof}

\subsection{The Newton-Okounkov body of $\E$}

By the formula $H^0(X(\Sigma), \E) \cong H^0(\P\E, \O_\E(1))$ it is reasonable to regard the Newton-Okounkov body of the section ring of the line bundle $\O_\E(1)$ as the Newton-Okounkov body of the bundle $\E$. The divisor class of this bundle is $(0, 1) \in \CL(\P\E) \cong \CL(X(\Sigma))\times \Z$. Accordingly, we let $\Delta_C(\E)$ denote the Newton-Okounkov body of $\O(0, 1)$. 

Everything in this section on the global Newton-Okounkov body of $\P\E$ holds irrespective of the diagram we have chosen to represent the projectivization, however now some care must be taken as we are considering the vector bundle $\E$ itself in order to make use of Proposition \ref{prop-nobycayley}.  We assume that each column $D_j$ of the diagram defines an effective divisor class $\bd_j \in \CL(X(\Sigma))$.  We let $CP(\E)$ denote the Cayley polytope obtained by placing $\Delta(\bd_j)$ above the $j$-the vertex of the standard $m-1$ simplex. 

\begin{proposition}\label{prop-noby}
Let $\E$ be a toric vector bundle with $\RR(\P\E)$ generated in $\Sym$-degree $1$, suppose that $\bd_j$ effective for all $1 \leq j \leq m$, and suppose that there is maximal face $K \subset \Trop(L)$ such that $\pi^{-1}K \subset \Trop(\I)$ is a prime cone, then the Newton-Okounkov body $\Delta(\E)$ is the linear image of $CP(\E)$ under $M\circ\psi: M_\Q \times \Q^m \to \Q^{m + r}$.
\end{proposition}

\begin{proof}
This follows from applying Proposition \ref{prop-nobycayley} to $P_{0, 1}$. 
\end{proof}

\begin{example}\label{ex-tangentbundlenoby}

Let $X(\Sigma)$ be a smooth, projective toric variety, then the tangent bundle $\T X(\Sigma)$ is the sparse toric vector bundle with $D = I_n$ and linear ideal $L_\Sigma$ the ideal of relations which hold among the integral ray generators of the fan $\Sigma$. We apply Propositions \ref{prop-noby} and \ref{prop-NOKgeneralrow} to compute a global Newton-Okounkov body of $\P\T X(\Sigma)$ and the Newton-Okounkov body $\Delta(\T X(\Sigma))$.  

To build a suitable flag of flats it suffices to pick any maximal face $\sigma \in \Sigma$ and a total ordering of the rays $\sigma(1) = \{\rho_1, \ldots, \rho_d\}$.  We let $F_k \subset \MM(L_\Sigma)$ be the span of the first $k$ integral ray generators with respect to this ordering; then $F_1,\ldots, F_d = \MM(L_\Sigma)$ forms a full flag of flats. Let $M_\sigma: \Q^{n+m} \to \Q^{n+r}$ be the corresponding matrix as in Proposition \ref{prop-NOKgeneralrow}.  We then obtain a global Newton-Okounkov body $\Delta_\sigma = M_\sigma \Q_{\geq 0}^{n+m}$.

To compute $\Delta_\sigma(\T X(\Sigma))$ it only remains to describe $CP(\T X(\Sigma))$.  This Cayley polytope is formed from the polytopes $\Delta(\be_i)$ associated to the prime toric divisors of $X(\Sigma)$. The Newton-Okounkov body $\Delta_\sigma(\T X(\Sigma))$ is then $M_\sigma \circ \psi CP(\T X(\Sigma))$ as in Proposition \ref{prop-noby}.
\end{example}

\begin{example}\label{ex-Hilb2}

The procedure of this section for computing Newton-Okounkov bodies of projectivized toric vector bundles still applies with $\RR(\P\E)$ is finitely generated, but possibly by generaters of higher $\Sym$-degree.  We illustrate these constructions with the $2$nd symmetric power of the tangent bundle of $\P^2$.

The tangent bundle $\T\P^2$ can be constructed with $L = \langle y_0 + y_1 + y_2\rangle$ and $D = I_3$, the $3\times 3$ identity matrix. In particular, $\T\P^2$ is a sparse bundle. Now let $\Sym^2\T\P^2$ denote the $2$nd symmetric power of $\T\P^2$, this bundle can built with the following linear ideal and diagram:\\

\[\langle y_{11} +y_{12}+ y_{13}, y_{12}+y_{22}+y_{23}, y_{13}+y_{23} + y_{33} \rangle \subset \C[y_{12}, y_{13}, y_{23}, y_{11}, y_{22}, y_{33}],\]

\[D= \begin{bmatrix}
1 & 1 & 0 & 2 & 0 & 0\\
1 & 0 & 1 & 0 & 2 & 0\\
0 & 1 & 1 & 0 & 0 & 2\\
\end{bmatrix}\].\\

The corresponding elements $X_1, X_2, X_3, Y_{12}, Y_{13}, Y_{23}, Y_{11}, Y_{22}, Y_{33} \subset \RR(\P\Sym^2\T\P^2)$ do not form a generating set. However, after an application of \cite[Algorithm 5.6]{Kaveh-Manon-tvb} we find a single new generator $Z \in \RR(\P\Sym^2\T\P^2)$. The Cox ring $\RR(\P\Sym^2\T\P^2)$ is then generated by $X_1, X_2, X_3$ $Y_{12}, Y_{13}, Y_{23}, Y_{11}, Y_{22}, Y_{33}, Z$ subject to the ideal $\I$ generated by:  

\[X_1Y_{13}+X_2Y_{23}+X_3Y_{33}, \ \ X_2Y_{12}+X_1Y_{11}+X_3Y_{13}, \ \ X_1Y_{12}+X_2Y_{22}+X_3Y_{23},\] 

\[X_1X_2 Z+Y_{13}Y_{23}-Y_{12}Y_{33}, \ \ X_2X_3 Z+Y_{12}Y_{13}-Y_{23}Y_{11}, \ \ X_1X_3 Z+Y_{12}Y_{23}-Y_{13}Y_{22},\] 

\[X_2^2 Z-Y_{13}^2+Y_{11}Y_{33}, \ \ X_2^2 Z-Y_{12}^2+Y_{11}Y_{22}, \ \ X_1^2 Z-Y_{23}^2+Y_{22}Y_{33},\]

\[2Y_{12}Y_{13}Y_{23}-Y_{23}^2Y_{11}-Y_{13}^2Y_{22}-Y_{12}^2Y_{33}+Y_{11}Y_{22}Y_{33}.\]\\

\noindent
We extend the diagram $D$ by adding a column for the generator $Z$:

\[D'= \begin{bmatrix}
1 & 1 & 0 & 2 & 0 & 0 & 2\\
1 & 0 & 1 & 0 & 2 & 0 & 2\\
0 & 1 & 1 & 0 & 0 & 2 & 2\\
\end{bmatrix}\].\\

\noindent
We use Macaulay2 \cite{m2} to compute the tropical variety $\Trop(\I)$, and test each maximal face for prime points.  This results in $18$ out of the $20$ maximum dimensional cones being prime. Each of these $18$ cones has its own corresponding global Newton-Okounkov body of $\P\Sym^2\T\P^2$.  

To form the matrix $M$ for one of these prime cones, we begin with $[D', -I]$, and append a row which records the $\Sym$-degree of each generator. These rows correspond to grading of $\RR(\P\Sym^2\T\P^2)$ by $\Z^3\times \Z$, where $\Z^3$ can be viewed as $\CL(\P^2)$ crossed with the character lattice of $(\C^*)^2$. In particular, the rows of the resulting matrix belong to any maximal face of $\Trop(\I)$.  We finish the $M$ matrix by appending two more rows corresponding to a choice of particular prime cone $C$:

\[M = \begin{bmatrix}
1 & 1 & 0 & 2 & 0 & 0 & 2 & -1 & 0 & 0\\
1 & 0 & 1 & 0 & 2 & 0 & 2 & 0 & -1 & 0\\
0 & 1 & 1 & 0 & 0 & 2 & 2 & 0 & 0 & -1\\
1 & 1 & 1 & 1 & 1 & 1 & 2 & 0 & 0 & 0\\
4 & 4 & 4 & 4 & 4 & 4 & 17 & 0 & 0 & 0 \\
10 & 10 & 0 & 20 & 0 & 0 & 27 & 0 & 0 & 0\\    
\end{bmatrix}.\]\\

\noindent
The $\Q_{\geq 0}$-span of the columns of $M$ is the global Newton-Okounkov body $\Delta_C$. 

Now for $(\alpha, \beta) \in \Z \times \Z \cong \CL(\P\Sym^2\T\P^2)$ the polytope $P_{\alpha, \beta} \subset \Q^{6 + 1 + 3}$ is the collection of tuples $(y_{12}, y_{13}, y_{23}, y_{11}, y_{22}, y_{33}, z, x_1, x_2, x_3) \in \Q_{\geq 0}^{6 + 1 + 3}$ satisfying:\\

\[2(y_{12}+ y_{13}+ y_{23}+ y_{11}+ y_{22} + y_{33}) + 6z -x_1 -x_2 -x_3 = \alpha,\]
\[ y_{12}+ y_{13}+ y_{23}+ y_{11}+ y_{22} + y_{33} + 2z = \beta.\]\\

\noindent
We still have $\Delta_C(\alpha, \beta) = M\circ P_{\alpha, \beta}$ is the Newton-Okounkov body of $(\alpha, \beta) \in \CL(\P\Sym^2\T\P^2)$.
\end{example}

\section{Positivity properties}\label{sec-positivity}

In this section we study the pseudoeffective cone $\PsEff(\P\E)$, the effective monoid $\Eff(\P\E)$, the Nef cone (semiample cone) $\Nef(\P\E)$, and the basepoint-free monoid $\Bpf(\P\E)$ of a projectivized toric vector bundle $\P\E$.  Throughout, the spaces we deal with will be Mori dream spaces, usually projectivized toric vector bundles with $\RR(\P\E)$ generated in $\Sym$-degree $1$.  Working with this class of varieties has two useful consequences.  First, the semiample cone coincides with the Nef cone, and second, the cones we deal with are determined by the monoids in $\CL(\P\E)$ generated by the classes of certain generators of $\RR(\P\E)$.  

Let $X$ be a smooth Mori dream space with free class group $\CL(X)$.  For a point $p \in X$ we let $S_p \subset \CL(X)$ denote the monoid of classes $\bd \in \CL(X)$ which carry a section $s \in H^0(X, \O(\bd))$ with $s(p) \neq 0$. We let $C_p = \Q_{\geq 0} S_p \subset \CL(X)\otimes \Q$.  The following is straightforward, but is essential for making the computation of $\Nef(X)$ tractable. 

\begin{proposition}
With $X$ as above, suppose that $f_1, \ldots, f_m \in \RR(X)$ are a homogeneous generating set of the Cox ring, then we have $S_p = \Z_{\geq 0}\{\deg(f_i)\mid f_i(p) \neq 0\}$.  Moreover,\\

\[\Bpf(X) = \bigcap_{p \in X} S_p\]
\[\Nef(X) = \bigcap_{p \in X} C_p\]\\
\end{proposition}

Clearly only a finite number of distinct $S_p$ are possible, so the computation of $\Bpf(X)$ and $\Nef(X)$ can always be achieved with a finite number of points in $X$. When we consider a smooth toric variety $X(\Sigma)$ these points can always be chosen to be the $T$-fixed points.  We let $S_\sigma$ and $C_\sigma$ denote the monoid and cone associated to the torus fixed point defined by a maximal face of $\Sigma$.  Since $\Sigma$ is smooth, the monoid $S_\sigma$ is freely generated, and moreover any smooth polytope (say with inner normal fan $\Sigma$) is very ample (\cite[Proposition 2.4.4]{CLS}).  The following proposition summarizes what we need from these observations. 

\begin{proposition}\label{prop-toricpositivity}
Let $X(\Sigma)$ be a smooth, projective toric variety, then\\

\[\Bpf(X(\Sigma)) = \bigcap_{\sigma \in \Sigma} S_\sigma,\]

\[\Nef(X(\Sigma)) = \bigcap_{\sigma \in \Sigma} C_\sigma.\]\\

\noindent
Each monoid $S_\sigma$ is freely generated. As a consequence, $\Bpf(X(\Sigma))$ is a saturated monoid and any Nef class on $X(\Sigma)$ is basepoint free.  Moreover, any ample class on $X(\Sigma)$ is very ample. 
\end{proposition}

\subsection{Basics of $\Nef(\P\E)$}

In this section we give a combinatorial description of a finite set of points $p$ for which the monoids $S_p$ and cones $C_p$ suffice to compute $\Nef(\P\E)$ and $\Bpf(\P\E)$.  First we narrow the search to include only those points in the $T$-fixed point fibers of $\P\E$. 

\begin{proposition}\label{prop-Tfixedsuffices}
For $\Bpf(\P\E)$ and $\Nef(\P\E)$ it suffices to consider the points in the $T$-fixed-point fibers $(\P\E)_{\sigma} \subset \P\E$.
\end{proposition}

\begin{proof}
Let $p \in \P\E$, and let $n \in N$ be such that $p_0 = \lim_{t\to 0} \chi_n(t)\circ p \in \P\E_\sigma$ for a maximal face $\sigma \in \Sigma$. Now if $s \in \RR(\P\E)$ is a $T-$quasi-invariant section, we must have $s(p) = 0$ implies $s(p_0) = 0$.  As a consequence, any section which does not vanish at $p_0$ must not vanish at $p$, so $S_p \supseteq S_{p_0}$ and $C_p \supseteq C_{p_0}$. 
\end{proof}

By Proposition \ref{prop-Tfixedsuffices} it suffices to find a finite set of points $q \in \P\E_\sigma$ such that $S_q \subset S_p$ for every other $p \in \P\E_\sigma$. Let $L \subset \C[y_1, \ldots, y_m]$ be a linear ideal, where $y_1, \ldots, y_m$ give projective coordinates on the projective space $\P^{m-1}$, and let $V(L) \subset \P^{m-1}$ be the corresponding linear subspace.  The coordinate hyperplanes $H_i = V(y_i) \subset \P E$ define a stratification of $V(L)$ with closed strata $V_S(L) = V(L) \cap \bigcap_{i \in S} H_i$.  Let $\PP(L)$ denote the poset defined by these strata, where $V_S(L) \preceq V_{S'}(L)$ when $V_S(L) \subset V_{S'}(L)$. We are interested in the \emph{minimal} elements of $\PP(L)$, which in turn correspond to the maximal non-trivial flats of the matroid $\MM(L)$.  Observe that for any maximal flat $F$, $V_F(L)$ must be a single point.

\begin{definition}
For $\sigma \in \Sigma$ a maximal face, and $F$ a maximal flat of the matroid $\MM(\In_\sigma L)$, let $p_{\sigma, F} \in \P\E_\sigma$ be the unique point in $V_F(\In_\sigma L)$. Let $S_{\sigma, F}$ and $C_{\sigma, F}$ be the monoid and rational cone of nonzero sections associated to $p_{\sigma, F}$.
\end{definition}

\begin{theorem}\label{thm-nefCI}
Let $\E$ be toric vector bundle over a smooth toric variety $X(\Sigma)$, and suppose that the Cox ring $\RR(\P\E)$ is generated in $\Sym$-degree $1$, then\\

\[\Bpf(\P\E) = \bigcap_{\sigma, F} S_{\sigma, F},\] 

\[\Nef(\P\E) = \bigcap_{\sigma, F} C_{\sigma, F}.\]\\
\end{theorem}

\begin{proof}
First we observe that if $p \in \P\E_\sigma$, then the $y_j$ such that $y_j(p) = 0$ form a flat of $\MM(\In_\sigma L)$. It follows that the $Y_j \in \RR(\P\E)$ which are non-zero at $p$ contain the $Y_j$ from some maximal flat $F \subset \MM(\In_\sigma L)$. The Cox ring $\RR(\P\E)$ is generated by the $x_i$ $1 \leq i \leq n$ and $Y_j$ $1 \leq j \leq m$, hence $S_p$ is generated by $(-\be_i,0)$ for $i \notin \sigma(1)$ and $(\bd_j, 1)$ for $Y_j(p) \neq 0$. In turn, this implies that $S_p \supseteq S_{\sigma, F}$ and $C_p \supseteq C_{\sigma, F}$.
\end{proof}

\subsection{Monomial Bundles}

Theorem \ref{thm-nefCI} suggests that Nef classes are easier to control on projectivized toric vector bundles with tractable matroids $\MM(\In_\sigma L)$. The next definition is made with this idea in mind. 

\begin{definition}[Monomial Bundle]
A toric vector bundle $\E$ over a smooth toric variety $X(\Sigma)$ defined by $D \in \Delta(\Sigma, L)$ is said to be \emph{monomial} if each initial ideal $\In_\sigma L$ is a monomial ideal.      
\end{definition}

A bundle defined by $L$ and $D \in \Delta(\Sigma, L)$ is monomial precisely when, for each maximal face $\sigma \in \Sigma$, the minimal face of the Gr\"obner fan of $L$ containing the rows of $D$ corresponding to $\sigma$ is a maximal face.  

\begin{example}[Kaneyama bundles]
Any member of the the class of \emph{Kaneyama bundles} studied in Section \ref{sec-Fano} is monomial by Proposition \ref{prop-kaneyama-monomial}. In particular, any tangent bundle $\T X(\Sigma)$ of a smooth toric variety is monomial. 
\end{example}

Given a monomial bundle $\E$ over $X(\Sigma)$ which comes from the pair $(L, D)$, we show several ways to build a new monomial bundle $\E'$ coming from an extension $(L', D')$ as in Section \ref{sec-extension}.  Recall that we realize $L$ as $\ker(M)$, and $L'$ as $\ker(M')$ with $M' = [M \mid U]$, and $D' = [D\mid X]$. 

\begin{theorem}\label{theorem: CI monomial extension}
    Let $\E$ be a monomial and CI bundle. Then $\E'$ is monomial and CI bundle if $x_{ij} \in X$ is larger than the minimal element of any circuit of $M$.
\end{theorem}

\begin{proof}
    We use the condition for CI described in Theorem \ref{thm-CIcondition}. Note that, if the circuit condition holds, then the minimal columns of $M'$ are always contained in the $M$ part. Therefore, since $M$ is CI and adding columns adds equally to both the support and the rank, we get that $M'$ is CI. 
    Similarly, since our initial forms are taken with respect to weights with the results ordered minimally, adding larger elements to a row does not impact the initial forms. Therefore, $\E'$ is monomial.
\end{proof}

For the next few results, we specialize extended toric vector bundles to the uniform case. Then, we can not have $U$ be any matrix; rather, let $M' = [M \mid U]$ where the entries of $U$ still allow $M'$ to be very general.

\begin{lemma}
Let $\E$ be a uniform toric vector bundle corresponding to $(L,D)$ such that that $D$ has no zero columns. Then $\E'$ is a uniform monomial toric vector bundle if and only if every choice of $r-1$ rows of $D$ has exactly one zero column.
\end{lemma}

\begin{proof}
By Proposition 3.9, $\E'$ corresponding to $(L',D')$ is CI. Therefore, we need only check that $\text{in}_\sigma(L')$ is monomial for all maximal $\sigma \in \Sigma$. Since $\E$ is uniform, $\Sigma(1) = e_j$ for $1 \leq j \leq m$ and the maximal cones $\sigma \in \Sigma$ are generated by $r-1$ such rays. Since $D$ has no zero columns, every choice of $r-1$ rows of $D$ corresponds to a maximal cone of $\Sigma$. 
\end{proof}

\begin{theorem}
Let $\E$ be a uniform monomial toric vector bundle. Then $\E'$ is a monomial toric vector bundle if and only if $X \geq 0$ has no more than $r-2$ zero entries per column.
\end{theorem}

\begin{proof}
By the above lemma, it suffices to confirm that every choice of $r-1$ rows of $D'$ has exactly one zero column. Note that this is only possible if $X$ has no column with more than $r-2$ zero entries since, if there were, choosing these rows would introduce another zero column. 
\end{proof}

Now that we have established that the class of monomial bundles $\E$ with $\RR(\P\E)$ is sufficiently rich, we come to the main result of this section. In what follows we let $F_j = \{y_1, \ldots, y_m\}\setminus \{y_j\}$.

\begin{proposition}\label{prop-monomial}
Let $\E$ be a monomial bundle over $X(\Sigma)$ with $\RR(\P\E)$ generated in $\Sym$-degree $1$, then: \\

\begin{enumerate}
\item The maximal flats of $\MM(\In_\sigma L)$ are all of the form $F_j$ for some $j$,
\item The monoids $S_{\sigma, F}$ are all freely generated, 
\item The cones $C_{\sigma, F}$ are all smooth, 
\item Any Nef class on $\P\E$ is basepoint-free, 
\item $\P\E$ satisfies Fujita's freeness conjecture.\\ 
\end{enumerate}

\end{proposition}

\begin{proof}
First we observe that if each $\In_\sigma L$ is a monomial ideal, then $\MM(\In_\sigma L)$ is a single basis with $m-r$ loops.  As a consequence, any maximal flat of $\MM(\In_\sigma L)$ is the complement of a single element.  We see then that $S_{\sigma, F} = -S_\sigma \times \{0\} + \Z_{\geq 0}(\bd_j, 1)$ for some $1 \leq j \leq m$.  It follows that $S_{\sigma, F}$ is freely generated, this proves $1$, $2$, and $3$.  Moreover, $\Bpf(\P\E)$ is the intersection of saturated monoids, and is therefore saturated.  This proves $3$, $4$, and $5$.   
\end{proof}

\subsection{Toric Vector Bundles over $\P^{\nn}$}

If the Cox ring $\RR(\P\E)$ of a projectivized toric vector bundle is generated in $\Sym$-degree $1$, Theorem \ref{thm-nefCI} implies $\Nef(\P\E)$ is an intersection of cones of the form $C_\sigma\times\{0\} + \sum \Q_{\geq 0}(\bd_j, 1)$.  Computing such an intersection simplifies when each cone $C_\sigma$ coincides with the Nef cone of $X(\Sigma)$; this is precisely what happens when $X(\Sigma) = \P^{\nn} = \P^{n_1} \times \cdots \times \P^{n_\ell}$.  

Recall that a \emph{coloop} of a matroid $\MM$ is an element whose complement is a flat.  Coloops play an important role in computing the Nef cone of a monomial bundle, they are also the key player for positivity properties of toric vector bundles bundles over $\P^{\nn}$. 

\begin{theorem}\label{thm-Pn}
Let $\E$ be a toric vector bundle over $\P^{\nn}$ with $\RR(\P\E)$ generated in $\Sym$-degree $1$, and suppose further that for any $1 \leq j \leq m$ there is a maximal face $\sigma$ such that $y_j$ is a coloop of $\MM(\In_\sigma L)$, then:\\

\begin{enumerate}
\item Any Nef class on $\P\E$ is basepoint-free, 
\item Any ample class on $\P\E$ is very ample, 
\item $\P\E$ satisfies Fujita's freeness and ampleness conjectures. \\
\end{enumerate}

\end{theorem}

\begin{proof}
It suffices to prove $1$ and $2$.  We consider the split toric vector bundle $\V_D = \bigoplus_{j =1}^m \O(D_j)$ determined by the columns of the diagram $D$.  We have a surjection $\F_D \to \E$ and an embedding $\P\E \to \P\F_D$, and the Cox ring $\RR(\P\F_D)$ presents $\RR(\P\E)$ via the generators $X_i$, $Y_j$. As a consequence, we get inclusions $\Bpf(\P\V_D) \subseteq \Bpf(\P\E)$ as subsets of $\CL(\P\V_D) \cong \CL(\P\E)$.  Our assumption implies that $\Bpf(\P\E) = \bigcap_{j =1}^m (-\Nef(\P^{\nn}), 0) + \Z_{\geq 0}(\bd_j, 1) = \Bpf(\P\V_D)$. This directly implies $1$.  Also, it follows that any Nef class or basepoint-free class on $\P\E$ extends to a Nef or basepoint-free class on $\P\V_D$.  The space $\P\V_D$ is a smooth toric variety, so any ample class is very ample (in fact projectively normal in this case), so the same must be true for $\P\E$.  This proves $2$ and $3$.  
\end{proof}

The argument we use to prove Theorem \ref{thm-Pn} works any time a Mori dream space $X$ has a neat embedding into a smooth toric variety $Z$ inducing an equality $\Bpf(X) = \Bpf(Z)$ in $\CL(X) = \CL(Z)$.  We propose that an embedding with this property be called \emph{neat and tidy}. The next lemma provides a combinatorial check for the coloop property in the hypthosesis of Theorem \ref{thm-Pn}.

\begin{lemma}\label{lem-point}
Let $\E$ be toric vector bundle over $X(\Sigma)$ associated to $(L, D)$, and suppose that every circuit of the ideal $L$ has at least one $0$ entry in each row of $D$. Then if $D_j \neq 0$ there is some maximal $\sigma \in \Sigma$ such that $y_j$ is a coloop of $\MM(\In_\sigma L)$. 
\end{lemma}

\begin{proof}
We will show that there must be some $\sigma \in \Sigma$ for which $y_j$ is not supported on any linear polynomial in $\In_\sigma(L)$.  If this is the case, we can find a standard basis for $\In_\sigma(L)$ containing $y_j$, and define $p$ by setting $y_j(p) = 1, y_k(p) = 0$ $k \neq j$.

Observe that if $y_j$ is supported on a linear form in $\In_\sigma(L)$ then $y_j$ must be supported on some linear form in every row of $D$ corresponding to the elements of $\sigma(1)$.  It follows that if $y_j$ is supported on a linear form in each $\In_\sigma(L)$, then $y_j$ is supported on at least one linear form in the initial ideal of each row of $D$.  Moreover, the circuits of $L$ form a universal Gr\"obner basis, so we conclude that for each row $\beta$ of $D$ there is a circuit $C_\beta \in L$ such that $y_j$ is supported on the initial form $\In_\beta(C_\beta)$.  But the minimum entry in a row supported on a circuit must be $0$.  We conclude that $D_j = 0$, which is a contradiction.
\end{proof}

The diagram condition in Lemma \ref{lem-point} is satisfied for both sparse bundles and uniform bundles, so we obtain the following corollary.

\begin{corollary}\label{cor-productspositivity}
Let $\E$ be a sparse toric vector bundle or a uniform toric vector bundle over $\P^{\nn}$ with $\RR(\P\E)$ generated in $\Sym$-degree $1$, then:\\

\begin{enumerate}
\item Any Nef class on $\P\E$ is basepoint-free, 
\item Any ample class on $\P\E$ is very ample, 
\item $\P\E$ satisfies Fujita's freeness and ampleness conjectures.\\ 
\end{enumerate}
\end{corollary}

\subsection{Toric Vector bundles over $\P^n$}

Projectivized toric vector bundles $\P\E$ over $\P^n$ have Picard rank $2$, so it can be expected that they have even better behavior than bundles over more general toric varieties.  However, it is not always the case that projectivized toric vector bundles over $\P^n$ are Mori dream spaces, for a non-example of a rank $3$ bundle over $\P^2$ see \cite[Example 6.9]{Kaveh-Manon-tvb}.  

If we assume that $\RR(\P\E)$ is generated in $\Sym$-degree $1$, then any monoid $S_{\sigma, F}$ is of the form $\Z_{\geq 0}\{(-1, 0), (d_j, 1)\}$ for $d_j \in \Z_{\geq 0}$.  As a consequence, the effective monoid of any such bundle is generated by $(-1, 0)$ and $(d, 1)$, where $d$ is the maximum of the integers which appear.  Similarly, the basepoint-free monoid of $\P\E$ is generated by $(-1, 0)$ and the minimum of the $d_j$.  

\begin{proposition}\label{prop-peff-bpf-Pn}
Let $\E$ be a toric vector bundle over $\P^n$ with $\RR(\P\E)$ generated in $\Sym$-degree $1$, then any pseudo-effective class on $\P\E$ is effective, and any Nef class on $\P\E$ is basepoint-free.  As a consequence, $\P\E$ satisfies Fujita's freeness conjectures. 
\end{proposition}

\begin{proof}
The statements about pseudo-effective and Nef classes following from the discussion above. 
\end{proof}

\begin{example}\label{ex-Hilb2again}
We return to the example of the $2$nd symmetric power of the tangent bundle of $\P^2$ because this bundle shows that the assumption that $\RR(\P\E)$ be generated in $\Sym$-degree $1$ is necessary.  Recall that $\RR(\P\Sym^2\T\P^2)$ requires a generator of $\Sym$-degree $2$. The bundle $\Sym^2\T\P^2$ is rank $3$ over $\P^2$, and the degrees of these generators in $\CL(\P\Sym^2\T_2) \cong \Z \times \Z$ are $\deg(X_i) = (-1, 0), \deg(Y_{ij}) = (2, 1), \deg(Z) = (6, 2)$. In particular, the class $(3, 1) = \frac{1}{2}\deg(Z)$ lies in the pseudo-effective cone but is not effective.  

It can be shown that $\P\Sym^2\T\P^2$ is isomorphic to the Hilbert scheme $\textup{Hilb}^2(\P^2)$ of pairs of points on $\P^2$.  The $0$-locus of $Z$ corresponds to the exceptional divisor of $\textup{Hilb}^2\P^2$ when it is viewed as a blow-up of the symmetric power $\Sym^2\P^2$. 
\end{example}

\begin{example}
This example shows that the assumption that $\P\E$ be a projectivized bundle over a single projective space $\P^n$ is necessary for the effectiveness statement in Proposition \ref{prop-peff-bpf-Pn}.  We consider a rank $2$ bundle over $\P^1\times\P^1$ with $L = \langle y_1+y_2+y_3+y_4 \rangle$ and 
    $$
    D = \begin{bmatrix}
        a_1 & 0 & 0 & 0 \\
        0 & a_2 & 0 & 0 \\
        0 & 0 & a_3 & 0 \\
        0 & 0 & 0 & a_4
    \end{bmatrix}.
    $$ 
We then map 
    $$
\begin{aligned}
\deg(Y_1) =  (a_1,0,1) \\
\deg(Y_2) =  (0,a_2,1) \\
\deg(Y_3) =  (a_3,0,1) \\
\deg(Y_4) =  (0,a_4,1)
\end{aligned}
    $$  
    and have the cones
\\
    
\begin{figure}[H]
\begin{center}
\begin{tikzpicture}[scale=.75,transform shape]
    \tkzInit[xmax=5,ymax=5,xmin=-5,ymin=-5]
    \tkzGrid
   % \tkzAxeXY
   \draw[thick] (-5,0) -- (5,0);
   \draw[thick] (0,5) -- (0,-5);
   \draw[dashed] (0,4) -- (4,0);
    \draw[very thick] (-5,0) -- (2,0) -- (2,-5) node[] at (2,0.25) {$a_1$};
    \draw[very thick] (-5,0) -- (4,0) -- (4,-5) node[] at (4,0.25) {$a_3$};
    \draw[very thick] (-5,2) -- (0,2) -- (0,-5) node[] at (.35, 2) {$a_2$};
    \draw[very thick] (-5,4) -- (0,4) -- (0,-5)  node[] at (.35, 4) {$a_4$};
    \fill[orange,opacity=0.2] (-5,0) --(2,0) -- (2,-5) -- (-5,-5) -- cycle;
    \fill[blue,opacity=0.2] (-5,0) -- (4,0) -- (4,-5) -- (-5,-5) -- cycle;
    \fill[red, opacity=0.2] (-5,2) -- (0,2) -- (0,-5) -- (-5,-5) -- cycle;
    \fill[green, opacity=0.2] (-5,4) -- (0,4) -- (0,-5) -- (-5,-5) -- cycle;
   
  \end{tikzpicture}
  \caption{The class group, $\CL(\P\E)$.}
  \end{center}
  \end{figure}

Where the negative orthant is the (saturated!) Nef cone. However, the integral points on the line connecting $a_3$ and $a_4$ represent classes that are pseudoeffective but not effective.
\end{example}

\subsection{Example where Nef does not imply basepoint free}

We give an example of a sparse toric vector bundle $\F$ (Definition \ref{def-sparse}) over a smooth toric surface $Y(\Sigma)$ whose projectivization $\P\F$ has Nef classes which are not basepoint free.  As a reminder, $\F$ cannot be monomial, and $Y(\Sigma)$ cannot be $\P^2$ or $\P^1\times \P^1$.  We let $\Sigma$ be the smooth fan in $\Q^2$ with integral ray generators $(1,0), (1,1), (1, 2), (0, 1), (-1,0), (0, -1)$, in other words the blow up at a toric fixed point of a blow up at a toric fixed point of $\P^1\times \P^1$.  We let $\F$ be the rank $2$ bundle on $Y(\Sigma)$ corresponding to the following data:

\[L = \langle y_1 + y_2 + y_3 \rangle \subset \C[y_1, y_2, y_3],\]

\[D = \begin{bmatrix} 
c_1 & 0 & 0\\
c_2 & 0 & 0\\
c_3 & 0 & 0\\
c_4 & 0 & 0\\
0 & c_5 & 0\\
0 & 0 & c_6
\end{bmatrix}\]\\

\noindent
Where $c_1 = 3, c_2 = 6, c_3 = 9, c_4 = 2, c_5 = 9, c_6 = 6$. We present the toric vector bundle in this way to emphasize that the $c_i$ can be given alternative non-negative values while maintaining properties of $\P\F$. Toric vector bundles of this form have previously found use as counterexamples to global generation in work of N\o dland \cite{Nodland}.  The Cox ring $\RR(\P\F)$ is presented as follows:

\[\C[x_1,x_2,x_3,x_4,x_5,x_6, Y_1,Y_2,Y_3]/\langle x_1^3x_2^6x_3^9x_4^2Y_1 + x_5^9Y_2 + x_6^6Y_3\rangle.\]\\

We let the classes of the toric divisors corresponding to the rays of $\Sigma$ be denoted $e_i \in \CL(Y(\Sigma))$ $1 \leq i \leq 6$.  The class group $\CL(Y(\Sigma))$ is freely generated by $e_1, e_2, e_3, e_4$ with $e_5 = e_1 + e_2 + e_3$ and $e_6 = e_2 + 2e_3 + e_4$. As a consequence, the generators of $\RR(\P\F)$ have the following classes in $\CL(\P\F) \cong \CL(Y(\Sigma)) \times \Z$: $[x_i] = (-e_i, 0)$, $[Y_1] = (3e_1 + 6e_2 + 9e_3 + 2e_4, 1)$, $[Y_2] = (9e_1 + 9e_2 + 9e_3, 1)$, $[Y_6] = (6e_2 + 12e_3 + 6e_4, 1)$.

We label the cones of $\Sigma$ counter clockwise: $\sigma_1 = \Q_{\geq 0}\{(1,0),(1,1)\}$, $\sigma_2 = \Q_{\geq 0}\{(1,1),(1,2)\}$, $\sigma_3 = \Q_{\geq 0}\{(1,2),(0,1)\}$, $\sigma_4 = \Q_{\geq 0}\{(0,1),(-1,0)\}$, $\sigma_5 = \Q_{\geq 0}\{(-1,0),(0,-1)\}$, and $\sigma_6 = \Q_{\geq 0}\{(0,-1),(1,0)\}$.

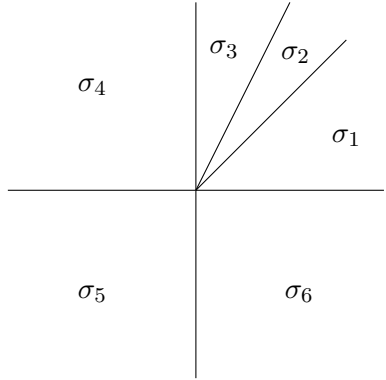
\vspace{.15in}

    \begin{center}
\begin{tikzpicture}[scale = 2.5]

%\draw (0,0) -- (1.65,0) -- (0,0) -- (0,1.65) -- (0,0) -- (-1,-1) -- (0,0)
\draw (0,0) -- (1,0) node[] at (0.8,0.28) {$\sigma_1$};
\draw (0,0) -- (.8,.8) node[] at (0.53,0.72){$\sigma_2$};
\draw (0,0) -- (.5,1) node[] at (.15, .75) {$\sigma_3$};
\draw (0,0) -- (0,1) node[] at (-.55,.55) {$\sigma_4$}; 
\draw (0,0) -- (-1,0) node[] at (-.55, -.55) {$\sigma_5$};
\draw (0,0) -- (0,-1) node[] at (.55,-.55) {$\sigma_6$};
\end{tikzpicture}
\end{center}

\vspace{.15in}

Each initial ideal $\In_{\sigma_i}L$ has two minimal elements which we denote with a $0,1$ vector indicating the support of the element.  For example, over $\sigma_1$ the initial ideal is $\In_{\sigma_1}L = \langle y_2 + y_3\rangle$, which has minimally supported solution types $100$ (e.g. $(3, 0,0)$ and $011$ (e.g. $(0, -5, 5)$).  There are $12$ corresponding monoids.

\[S_1^{100} = \langle (-e_3,0),(-e_4,0),(-e_5,0),(-e_6,0), (3e_1 + 6e_2 + 9e_3 + 2e_4, 1) \rangle \]

\[S_1^{011} = \langle (-e_3,0),(-e_4,0),(-e_5,0),(-e_6,0), (9e_1 + 9e_2 + 9e_3, 1), (6e_2 + 12e_3 + 6e_4, 1) \rangle\]

\[S_2^{100} = \langle (-e_1,0),(-e_4,0),(-e_5,0),(-e_6,0), (3e_1 + 6e_2 + 9e_3 + 2e_4, 1) \rangle \]

\[S_2^{011} = \langle (-e_1,0),(-e_4,0),(-e_5,0),(-e_6,0), (9e_1 + 9e_2 + 9e_3, 1), (6e_2 + 12e_3 + 6e_4, 1) \rangle\]

\[S_3^{100} = \langle (-e_1,0),(-e_2,0),(-e_5,0),(-e_6,0), (3e_1 + 6e_2 + 9e_3 + 2e_4, 1) \rangle \]

\[S_3^{011} = \langle (-e_1,0),(-e_2,0),(-e_5,0),(-e_6,0), (9e_1 + 9e_2 + 9e_3, 1), (6e_2 + 12e_3 + 6e_4, 1) \rangle\]

\[S_4^{100} = \langle (-e_1,0),(-e_2,0),(-e_3,0),(-e_6,0), (3e_1 + 6e_2 + 9e_3 + 2e_4, 1) \rangle \]

\[S_4^{010} = \langle (-e_1,0),(-e_2,0),(-e_3,0),(-e_6,0), (9e_1 + 9e_2 + 9e_3, 1) \rangle\]

\[S_5^{010} = \langle (-e_1,0),(-e_2,0),(-e_3,0),(-e_4,0), (9e_1 + 9e_2 + 9e_3, 1), \rangle \]

\[S_5^{001} = \langle (-e_1,0),(-e_2,0),(-e_3,0),(-e_4,0), (6e_2 + 12e_3 + 6e_4, 1) \rangle\]

\[S_6^{100} = \langle (-e_2,0),(-e_3,0),(-e_4,0),(-e_5,0), (3e_1 + 6e_2 + 9e_3 + 2e_4, 1), \rangle \]

\[S_6^{001} = \langle (-e_2,0),(-e_3,0),(-e_4,0),(-e_5,0), (6e_2 + 12e_3 + 6e_4, 1) \rangle\]\\

\noindent
Every monoid except $S_1^{011}, S_2^{011}, S_3^{011}$ is smooth.  Intuitively, it is the distance between the rays generated by $[Y_2]$ and $[Y_3]$ which is allowing for integral points in the convex hulls of these monoids to be missed. 

We let $C_i^{***}$ be the cone generated by $S_i^{***}$ so that $\Nef(\P\F) = \bigcap C_i^{***}$ and $\Bpf(\P\F) = \bigcap S_i^{***}$. We get the following vectors as the generators of our Hilbert basis:
$$
\begin{bmatrix} -2 \\ -2 \\ -2 \\ -1 \\ 0 \end{bmatrix}, 
\begin{bmatrix}  0 \\ 5 \\ 10 \\ 0 \\ 2 \end{bmatrix},
\begin{bmatrix} 0 \\ 1 \\ 2 \\ 0 \\ 1 \end{bmatrix},
\begin{bmatrix} 0 \\ 2 \\ 3 \\ 0 \\ 1 \end{bmatrix},
\begin{bmatrix} 0 \\ 3 \\ 5 \\ 0 \\ 1 \end{bmatrix},
\begin{bmatrix} -1 \\ -1 \\ -1 \\ 0 \\ 0 \end{bmatrix},
\begin{bmatrix} 0 \\ 2 \\ 4 \\ 0 \\ 1 \end{bmatrix},
\begin{bmatrix} -1 \\ 2 \\ 5 \\ 0 \\ 1 \end{bmatrix}
\begin{bmatrix} 0 \\ 3 \\ 3 \\ -1 \\ 1 \end{bmatrix},
\begin{bmatrix}  -1 \\ -1 \\ -2 \\ -1 \\ 0 \end{bmatrix},
\begin{bmatrix} 0 \\ -1 \\ -2 \\ -1 \\ 0 \end{bmatrix},
\begin{bmatrix}  0 \\ 3 \\ 4 \\ 0 \\ 1 \end{bmatrix}.
$$

%| -2 |, | 0  |, | 0 |, | 0 |, | 0 |, | -1 |, | 0 |, | -1 |, | 0  |, | -1 |, | 0  |, | 0 |
%| -2 |  | 5  |  | 1 |  | 2 |  | 3 |  | -1 |  | 2 |  | 2  |  | 3  |  | -1 |  | -1 |  | 3 |
%| -2 |  | 10 |  | 2 |  | 3 |  | 5 |  | -1 |  | 4 |  | 5  |  | 3  |  | -2 |  | -2 |  | 4 |
%| -1 |  | 0  |  | 0 |  | 0 |  | 0 |  | 0  |  | 0 |  | 0  |  | -1 |  | -1 |  | -1 |  | 0 |
%| 0  |  | 2  |  | 1 |  | 1 |  | 1 |  | 0  |  | 1 |  | 1  |  | 1  |  | 0  |  | 0  |  | 1 |
     
We the check which, if any, of the basis elements is not able to be written as an integral combination of the ray generators that define each of our unsaturated cones, $C_1^{011}, C_2^{011}, C_3^{011}$. We get that the pairs 
$$
\left \{ \begin{bmatrix} 0 \\ -1 \\ -2 \\ -1 \\ 0 \end{bmatrix}, C_1^{011} \right \} \quad 
\left \{ \begin{bmatrix}  0 \\ 5 \\ 10 \\ 0 \\ 2 \end{bmatrix}, C_2^{011} \right \} \quad 
\left \{ \begin{bmatrix} 0 \\ -1 \\ -2 \\ -1 \\ 0 \end{bmatrix}, C_3^{011} \right \}. 
$$
have exactly this issue. Therefore, these classes are not basepoint free.

\section{Fano projectivized toric vector bundles}\label{sec-Fano}

We have shown that $\Nef(\P\E)$ can be computed directly when $\E$ is a monomial bundle with $\RR(\P\E)$ generated in $\Sym$-degree 1.  Moreover, when $\RR(\P\E)$ is a complete intersection it is easy to compute the anti-canonical class. With these ingredients, we have what we need to test if $\P\E$ Fano. In this section we obtain a complete classification of Fano \emph{Kaneyama} bundles, see Definition \ref{def-genkaneyama}.  As a consequence we find a new proof of a result of Mori \cite{Mori} in the toric case (Remark \ref{rem-Mori}).

\subsection{(non)-Fano diagonal bundles}

Let $\pi: \E \to X(\Sigma)$ be the toric vector bundle associated to the pair $(L, D)$, where $D$ is a diagonal matrix.  The next Proposition characterizes when a fan $\Sigma$ and a fixed linear ideal $L$ admit a diagonal bundle. 

\begin{proposition}\label{prop-diagonal}
Fix a linear ideal $L \subset \C[y_1, \ldots, y_m]$, an $m \times m$ diagonal matrix $D$ with positive diagonal entries, and a complete polyhedral fan $\Sigma \subset N_\Q$.  The data $(L, D)$ defines a toric vector bundle over $Y(\Sigma)$ if and only if for every collection $\sigma(1) \subset \Sigma(1)$ the corresponding $\{y_i \mid i \in \sigma(1)\}$ form an independent set of $\MM(L)$. 
\end{proposition}

\begin{proof}
 Let $\sigma \in \Sigma$ be a face with rays $\rho_1, \ldots, \rho_d \in \sigma(1)$. Let $R_1, \ldots, R_d$ be the corresponding rows of $D$.  Each $R_i$ contains a single non-zero entry; without loss of generality we suppose that the non-zero entry of $R_i$ occurs in the $i$-th position along the row.  Suppose that $D \in \Delta(L, \Sigma)$, then $R_1, \ldots, R_d$ are all contained in a common apartment of $\Trop(L)$. We claim that $\In_\sigma(L)$ is generated by circuits which do not contain $y_1, \ldots, y_d$.  The ideal $\In_\sigma(L)$ is itself an initial ideal of $\In_{R_i}(L)$ for any $i$.  But the circuits of $\In_{R_i}(L)$ cannot contain $y_i$. This is the case for $1 \leq i \leq d$. It follows that $y_1, \ldots, y_d$ are independent.  Now suppose that for every $\sigma \in \Sigma$ the $\{y_i \mid i \in \sigma(1)\}$ form an independent set. Let $\B$ be a basis which contains this set, and consider the maximal face $C_\B$ of the Gr\"obner fan of $L$ corresponding to $\B$. We claim $R_i \in C_\B$ for each $i \in \sigma(1)$. This amounts to checking that the value of each $y_j$ given by the prevaluation defined by $R_i$ is computable on the $\B$-expression for $y_j$. We've assumed that $L$ contains no binomials or monomials, so any expression for $y_j$ $j \neq i$ involves an element other than $y_i$.  As the only non-zero entry of $R_i$ is at the $i$-th spot, this means $y_j$'s value is the minimum of the values of the elements appear in its $\B$-expression.  This proves that $R_i$ $i \in \sigma(1)$ all belong to the apartment of $\B$.  
\end{proof}

If $D$ is diagonal and $rank(\E) = m$, then it is straightforward to see that $\E$ is a split toric vector bundle.  At the other extreme are the tangent bundles $\T X(\Sigma)$ of smooth toric varieties, where $D$ is the $n \times n$ identity matrix.  The next result is concerned with bundles of this latter form. 
 
\begin{definition}\label{def-genkaneyama}
We say $\E$ is a Kaneyama bundle if $\E$ has rank equal to the dimension of $X(\Sigma)$, and has non-negative diagram a diagonal matrix. 
\end{definition}

\noindent
The terminology ``Kaneyama bundle'' is in reference to the work of Kaneyama \cite{Kaneyama88}, where it is shown that half of the irreducible toric vector bundles of rank $n$ over $\P^n$ are of this form. The other half are the dual bundles of Kaneyama bundles. 

By Proposition \ref{prop-diagonal}, an ideal $L$ and a (positive) diagonal matrix $D$ defines a Kaneyama bundle over $X(\Sigma)$ precisely when $\{y_i \mid i \in \sigma(1)\}$ is a basis of $\MM(L)$ for every maximal face $\sigma \in \Sigma$. For now on we fix a Kaneyama bundle $\E$ and let $d_1 \leq \cdots \leq d_n$ denote the values along the diagonal of $D$. Any Kaneyama bundle is sparse, so it is straightforward to get an expression for the anticanonical class of $\P\E$:\\

\[-K_{\P\E} =  (\sum_{i =1}^n (d_i -1) \be_i, r) \]\\

\noindent
We wish to know when $-K_{\P\E} \in \Nef(\P\E)$. The following cones will be useful for this computation. Let $i \in \sigma(1)$ and let \\

\[S_{\sigma, j} = \sum_{i \notin \sigma(1)} \Q_{\geq 0}(\be_i, 0) + \Q_{\geq 0}(d_j\be_j, 1) \subset \CL(\P\E)_\Q.\]\\

\begin{proposition}\label{prop-kaneyama-monomial}
Let $\E$ be a Kaneyama bundle over $X(\Sigma)$ associated to the positive diagonal matrix $D \in \Delta(\Sigma, L)$, then $\E$ is a monomial bundle, and: 

\[\Nef(\P\E) = \bigcap_{i \in \sigma(1)} S_{\sigma, i}.\]\\

\noindent
In particular, $\Bpf(\P\E)$ is saturated, any Nef class is basepoint free, and $\P\E$ satisfies Fujita's freeness conjecture. 
\end{proposition}

\begin{proof}
We show that Kaneyama bundle is monomial, and that the set $\{y_i\mid i \in \sigma(1)\}$ forms the unique basis of $\MM(\In_\sigma L)$. Proposition \ref{prop-monomial} then implies the statement.   Fix a maximal face $\sigma$ and let $\tau$ be the face of the Gr\"obner fan $\GF(L)$ containing the rows $w_i = (0, \ldots, d_i, \ldots 0)$ for $i \in \sigma(1)$. It suffices to show that a $w \in \tau^\circ$ we have $y_j \in \In_w L$ for $j \notin \sigma(1)$. By assumption, the $y_i$, $i \in \sigma(1)$ form a basis of $\MM(L)$.  We let $y_j = \sum C_i y_i$ be the circuit expressing $y_j$ as a linear combination of the $y_i$ $i \in \sigma(1)$. We must have $\In_w(\In_{w_i} L) = \In_w L$, so for all $i \in \sigma(1)$, the initial form of $y_j - \sum C_i y_i$ cannot be supported on $y_i$, so $y_j \in \In_w L$. 
\end{proof}

If $-K_{\P\E}$ is to be a Nef class on $\P\E$ we must have that $-K_{\P\E}$ is contained in each $S_{\sigma, i}$.  This means that for each $i \in \sigma(1)$ there is an $m(i) \in M$ and $a_j \geq 0$ for $j \notin \sigma(1)$ such that:\\

\[\sum_{k =1}^n (d_k-1)\bee_k + \sum_{j \notin \sigma(1)} a_j \bee_j - rd_i \bee_i = \sum_{k =1}^n \langle \rho_k, m(i) \rangle \bee_k.\]\\

\noindent
where $\bee_k$ now denotes the $k$-th standard basis vector in $\Q^n$.

Now we assume $\Sigma$ is a smooth fan, $\sigma \in \Sigma$ is a maximal face.  We let $\epsilon_i \in M, i \in \sigma(1)$ be a dual basis to $\rho_i, i \in \sigma(1)$. For the sake of simplicity we assume $\sigma(1) = \{1, \ldots, d\}$.  Considering $1 \in \sigma(1)$, we write $m(1) = \sum m_\ell(1) \epsilon_\ell$.  Specializing to the coefficient of $\bee_1$ for $i \in \sigma(1)$ we get the following:\\

\[(d_1 -1) - rd_1 = m_1(1), \ \ (d_2 -1) = m_2(1), \ \ \ldots, (d_r-1) = m_r(1).\]\\

This leads to the matrix $X_\sigma$, with $\ell, i$-th entry $m_\ell(i)$:\\

$$X_\sigma = \begin{bmatrix}  
(1-r)d_1 -1  & d_2 -1 & \cdots & d_r -1 \\
d_1 -1 & (1-r)d_2 - 1 & \cdots & d_r -1 \\
\vdots & \vdots & & \vdots \\
d_1 -1 & d_2 -1 & \cdots & (1-r)d_r -1\\ 
\end{bmatrix}.$$\\

Now for $\rho_j$ $j \notin \sigma(1)$ we write $\rho_j = \sum t_i \rho_i$. The quantity $\langle \rho_j, m(i)\rangle$ is then the dot product of $(t_1, \ldots, t_r)$ with the $i$-th row of $X_\sigma$. This quantity equals $(d_j -1) + a_j \geq 0$, so we conclude that each $\rho_k$ lies in the cone with $H$-description $X_\sigma \bt \geq 0$.

\begin{proposition}\label{prop-nonnegative}
Any solution to $X_\sigma \bt \geq 0$ must be a member of the negative orthant. 
\end{proposition}

\begin{proof}
First we assume that $d_i \geq 2$.  Let $K$ be a diagonal matrix with all positive entries. Note that $K\bt$ lies in the negative orthant if and only if $\bt$ lies in the negative orthant.  Consequently, we can consider solutions to $X_\sigma K \bt \geq 0$; this reduces us to the following matrix:\\ 

$$W_\sigma = \begin{bmatrix}  
\alpha_1  & 1 & \cdots & 1 \\
1 & \alpha_2 & \cdots & 1 \\
\vdots & \vdots & & \vdots \\
1 & 1 & \cdots & \alpha_r\\ 
\end{bmatrix},$$\\

\noindent
where $\alpha_i = \frac{(1-r)d_i -1}{d_i -1} < 0$.  Let $\eta$ be an extremal ray of this cone, and suppose that all but the first entry of $W_\sigma \eta$ is $0$. Applying the difference of any two rows on $\eta$ then shows that the signs on the $2$-nd through $r$-th entries of $\eta$ must coincide. Comparing with the $1$-st entry shows that all entries of $\eta$ must be negative. 

Now, if say $d_1 = 1$, and all other entries are $\geq 2$, the argument above shows the signs of a solution all agree in entries $\geq 2$, so the same must be true for the first entry.  An induction argument now proves this for any number of $d_i = 1$. 
\end{proof}

\begin{corollary}\label{cor-nonnegative}
Let $\E$ be as above, and let $\Sigma$ be smooth, then if $-K_{\P\E} \in \Nef(\P\E)$ we must have that for any maximal face $\sigma \in \Sigma$ a ray $\rho_k$ not contained in $\sigma$ is contained in its negative. 
\end{corollary}

Corollary \ref{cor-nonnegative} places very strong restrictions on the fan $\Sigma$.  

\begin{lemma}\label{lem-isproduct}
Let $\Sigma$ be a smooth fan with the property that for any maximal face $\sigma \in \Sigma$ a ray $\rho_k$ not contained in $\sigma$ is contained in its negative, then $\Sigma$ is the fan of a product of projective spaces. 
\end{lemma}

\begin{proof}
Consider $e_1,...,e_n$ standard basis vectors which form a maximal face and $\rho = \sum a_ie_i$: 
$$
\begin{bmatrix}
1 & 0 & \hdots & 0 \\
0 & 1 & \hdots & 0 \\
 & & \ddots & \\
 0 & 0 & \hdots & 1 \\
 -\rho_1 & -\rho_2 & \hdots & -\rho_n
\end{bmatrix}.
$$
The collection $\{e_2,...,e_n,\rho\}$ also form a maximal face and, by the negative ray condition, 
$$
\pm 1 = \det \begin{bmatrix}
    e_2 \\ e_3 \\ \vdots \\ \rho
\end{bmatrix} = (-1)^{n+1}(-\rho_1), 
$$
so $\rho_1 = 1$ and any entries that were to appear in the matrix below $\rho_1$ would have to be zero. However, this omission can be done with any of the $e_i$, giving
$$
\begin{bmatrix} 
e_1 \\ e_2 \\ \vdots \\ e_n \\ -\rho_{s_1} \\ \vdots \\ -\rho_{s_l}
\end{bmatrix},
$$
where $s_1 \sqcup ... \sqcup s_l = [n]$. Then, we can regroup $\{e_{i_1}, ..., e_{i_k}, \rho_{s_i}\}$ where $i_j \in s_i$. These collections of rays form a maximal face and are exactly the maximal faces of a product of projective spaces.
\end{proof}

Now we directly compute $\Nef(\P\E)$ when $\E$ is a Kaneyama bundle over $\P^{\nn} = \P^{n_1} \times \P^{n_\ell}$ where $\ell > 1$.  Each $S_{\sigma, i}$ can be written $(-\Bpf(\P^{\nn}, 0) + \Z_{\geq 0}( (0, \ldots, d_i, \ldots, 0, 1)$.  let $\delta_s$ be the minimum of the $d_i$ corresponding to rays from a the $s$-th component of $\P^{\nn} = \P^{n_1}\times \cdots \times \P^{n_\ell}$, then we have:\\

\[\Bpf(\P\E) = \bigcap_{s =1}^\ell (-\Bpf(\P^{\nn}), 0) + \Z_{\geq 0}(0,\ldots, \delta_s, \ldots, 0, 1) = (-\Bpf(\P^{\nn}), 0) + \Z_{\geq 0}(0, \ldots, 0, 1).\]\\

\noindent
The $2$nd equality holds only if $\ell > 1$. The anticanonical class of $\P\E$ is:\\

\[-K_{\P\E} = (\sum_{i =1}^n (d_i -1) \be_i, r) =  (\ldots, (\sum d_i) - n_s, \ldots, \sum n_s).\]\\

\noindent
where $\sum d_i$ is over those rows belonging to the $s$-th component of $\P^{n_1}\times \ldots \times \P^{n_\ell}$.  This class is Nef if and only if $d_i = 1$ for all $i \in \Sigma(1)$, and this class is never ample as long as $\ell > 1$.  In particular, $\P\E$ is Fano only if $\E$ is a Kaneyama bundle over $\P^n$.

\begin{remark}\label{rem-Mori}[Tangent bundles and Mori's Theorem]
The tangent bundle $\T X(\Sigma)$ is a Kaneyama bundle for any smooth toric variety $X(\Sigma)$. The linear ideal $L_\Sigma$ is the ideal of relations which holds among the ray generators of $\Sigma$, and the matrix is the identity matrix. By above, the anticanonical class of $\P\T X(\Sigma)$ is then $-K_{\P\T X(\Sigma)} = (0, d) \in \CL(X(\Sigma))\times \Z$. Notably, the anticanonical class is a multiple of $\O_{\P\T X(\Sigma)}(1) = (0, 1) \in \CL(\P\T X(\Sigma))$.  Recall that a bundle $\E$ is said to be ample if $\O_\E(1)$ is ample.  A consequence of this section is then that the tangent bundle $\T X(\Sigma)$ is ample if and only if $-K_{\P\T X(\Sigma)}$ is ample, which in turn holds if and only if $X(\Sigma) = \P^n$.  This is the toric special case of Mori's proof \cite{Mori} of a conjecture of Hartshorne \cite{Hartshorne-ample}.  For a different combinatorial proof of Mori's theorem in the toric case, see \cite{Wu}. In fact, we prove that the only smooth, projective toric variety with a Nef tangent bundle is a product of projective spaces. 
\end{remark}

\begin{question}
    Let $X$ be a smooth, projective variety with Nef tangent bundle, is $X$ a product of projective spaces?
\end{question}

\subsection{Fano Kaneyama bundles over $\P^n$}

We finish this section by looking at the only remaining case: Kaneyama bundles over $\P^n$.  For rank reasons, the only linear ideal which can work is the hypersurface $L = \langle y_0 + \cdots + y_n \rangle$; so we reduce to the case of Kaneyama bundles.  Here there is still something to say. 

\begin{theorem}
Let $\E_\ba$ be the Kaneyama bundle over $\P^n$ corresponding to the non-negative integers $\ba = \{a_0, \ldots, a_n\}$ with $a_0 \leq \ldots \leq a_n$, then $\Eff(\P\E_\ba) = \Q_{\geq 0}\{(-1, 0), (a_n, 1)\}$, $\Nef(\P\E_\ba) = \Q_{\geq 0}\{(-1, 0), (a_0, 1)\}$, and $\P\E_\ba$ is Fano if and only if $\sum_{i =1}^n (a_i - a_0) \leq (n - a_0)$. 
\end{theorem}

\begin{proof}
The first two statements are direct computations. The anticanonical class of $\P\E_\ba$ is $(-n-1, 0) + (\sum a_i, n+1) - (0, 1) = (\sum (a_i -1), n)$. For this class to be ample we must have $\sum (a_i -1) < na_0$, or rather \\

\[a_0 -1 \leq \frac{1}{n+1} \sum (a_i -1) < \frac{n}{n+1}a_0.\]\\

For this to be able to hold we must have $a_0 < n+1$. Moreover, if we write $a_i = a_0 + \delta_i$, then this inequality holds exactly when $\sum_{i =1}^n \delta_i < n+1 - a_0$. 
\end{proof}

\bibliographystyle{alpha}
\bibliography{main}

\end{document}